\documentclass[11pt]{article}

\usepackage{amsmath,amssymb,amsthm,dsfont,mathrsfs,stmaryrd}
\usepackage[T1]{fontenc}
\usepackage[utf8]{inputenc}
\usepackage{titlesec,titletoc}
\usepackage[all,cmtip]{xy}
\usepackage[fleqn,tbtags]{mathtools}
\usepackage{mathabx}
\usepackage[numbers]{natbib}
\usepackage{tikz}
\usepackage{enumerate}  

\newtheorem{theoreme}{Theorem}[subsection]
\newtheorem{theo}{Theorem}[subsection]
\newtheorem{cor}{Corollary}[subsection]
\newtheorem{conj}{Conjecture}
\newtheorem{prop}[theo]{Proposition}
\newtheorem{lemme}[theo]{Lemma}
\theoremstyle{definition}
\newtheorem{defi}[theo]{Definition}
\theoremstyle{remark}
\newtheorem{rem}[theo]{Remark}
\newtheorem{ex}[theo]{Example}
\DeclareMathOperator{\supp}{supp}
\DeclareMathOperator{\vect}{Vect}

\DeclareMathOperator{\id}{Id}

\newcommand{\Uqg}{\mathcal{U}_{q}(\mathfrak{g})}
\newcommand{\Uqb}{\mathcal{U}_{q}(\mathfrak{b})}
\newcommand{\Uqbm}{\mathcal{U}_{q}^{-}(\mathfrak{b})}
\newcommand{\Uqbp}{\mathcal{U}_{q}^{\geq 0}(\mathfrak{b})}
\newcommand{\Uqbh}{\mathcal{U}_{q}(\mathfrak{b})^{0}}
\newcommand{\cO}{\mathcal{O}}
\newcommand{\Op}{\mathcal{O}^{+}}
\newcommand{\Om}{\mathcal{O}^{-}}
\newcommand{\OpZ}{\mathcal{O}^{+}_{\mathbb{Z}}}
\newcommand{\OmZ}{\mathcal{O}^{-}_{\mathbb{Z}}}
\newcommand{\PfN}{\mathcal{P}_{f}(\mathbb{N})}
\newcommand{\Plr}{P_{\ell}^{\mathfrak{r}}}

\numberwithin{equation}{section}
\author{Léa Bittmann}
\title{Asymptotics of standard modules of quantum affine algebras}
\date{}

\begin{document}

\maketitle

\begin{abstract} 
We introduce a sequence of $q$-characters of standard modules of a quantum affine algebra and we prove it has a limit as a formal power series. For $\mathfrak{g}=\hat{\mathfrak{sl}_{2}}$, we establish an explicit formula for the limit which enables us to construct corresponding asymptotical standard modules associated to each simple module in the category $\cO$ of a Borel subalgebra of the quantum affine algebra. Finally, we prove a decomposition formula for the limit formula into $q$-characters of simple modules in this category $\cO$.

\end{abstract}

\tableofcontents

	\section*{Introduction}

The quantized universal enveloping algebra of a finite-dimensional simple Lie algebra $\mathfrak{g}$ was introduced  independently by Drinfeld \citep{QGD} and Jimbo \citep{AqDAUg} in 1985. It is a $q$-deformation of the universal enveloping algebra $\mathcal{U}(\mathfrak{g})$, for $q$ a generic complex number, and it has a Hopf algebra structure.
If $\mathfrak{g}$ is an untwisted affine Kac-Moody algebra, then its universal enveloping algebra also admits a $q$-deformation, called the \textit{quantum affine algebra} $\Uqg$. The quantum affine algebra can also be obtained as a quantum affinization of the quantized universal enveloping algebra of the corresponding simple Lie algebra. Both processes (affinization then quantization and quantization then quantum affinization) commute (\citep{ANRY,BGA}). The quantum affine algebra has a presentation with the so-called "Drinfeld-Jimbo generators", and from this presentation one can define its Borel subalgebra $\Uqb \subset\Uqg$. 

Both the quantum affine algebra and the Borel algebra are Hopf algebras, thus their categories of finite-dimensional representations are monoidal categories. 
In \citep{QCRQAA}, Frenkel and Reshetikhin defined the $q$-character $\chi_{q}$, which is an injective ring morphism on the Grothendieck ring of the category of finite-dimensional representations of $\Uqg$, mimicking the usual character morphism. For every finite-dimensional representation $V$ of $\Uqg$, one can consider its $q$-character $\chi_{q}(V)$. It describes the decomposition of $V$ into $\ell$-weight spaces the action of the $\ell$-Cartan subalgebra of $\Uqg$. These $\ell$-weight are more suited than the weights spaces to the study of finite-dimensional $\Uqg$-modules (for detailed results, see \citep{CqC} for example).

Hernandez and Jimbo introduced in \citep{ARDRF} a category $\cO$ of representations of the Borel algebra $\Uqb$. Objects in this category $\cO$ are sums of their $\ell$-weight spaces and have finite-dimensional weight spaces, with weights in a finite union of certain cones. The category $\cO$ contains the finite-dimensional representations, as well as some other remarkable representations, called \textit{prefundamental representations}. For $\mathfrak{g}= \hat{\mathfrak{sl}_{2}}$, the latter appeared naturally, under the name $q$\textit{-oscillator representations} in Conformal Field Theory (for example in \citep{ISCFT}).
In general, the prefundamental representations are constructed as limits of particular sequences of simple finite-dimensional modules \citep{ARDRF}. Note that the prefundamental representations were used in \citep{BRSQIM} to prove the Frenkel-Reshetikhin's conjecture on the spectra of quantum integrable systems \citep{QCRQAA} (generalizing certain results of Baxter \citep{PF} on the spectra of the 6 and 8-vertex models).

The $q$-character morphism is a well-defined injective ring morphism, on the Grothendieck ring of the category $\cO$ (\citep{ARDRF}). 

In the category $\cO$, objects do not necessarily have finite lengths (see \citep[Lemma C.1]{HGS}, the tensor product of some prefundamental representations presents an infinite filtration of submodules), thus the category $\cO$ is not a Krull-Schmidt category. However, as in the case of the classical category $\cO$ of representations of a Kac-Moody Lie algebra, the multiplicity of a simple module in a module is well-defined. As such, all the elements of the Grothendieck ring of $\cO$ can be written as (infinite) sums of classes of simple modules.

There is a second basis of the Grothendieck ring of finite-dimensional $\Uqb$-modules formed by the \textit{standard modules}, introduced in \citep{QVFDR}. Both simple and standard modules are indexed by the same set: the monomials $m$ in the ring $\mathbb{Z}[Y_{i,a}]_{i\in I, a \in \mathbb{C}^{\times}}$. In \citep{QVFDR}, Nakajima showed that for the simply laced  types, the multiplicities of the simple modules in the standard modules can be realized as dimensions of certain varieties. Nakajima showed that the transition matrix between theses two bases is upper triangular: 
\begin{equation}\label{KL}
[M(m)] = [L(m)] + \sum_{m'< m} P_{m',m}[L(m')],
\end{equation}
for some partial order $\leq$ on the monomials in $\mathbb{Z}[Y_{i,a}]_{i\in I, a \in \mathbb{C}^{\times}}$, the exact definition of the order is recalled in (\ref{Nakajimaorder}), in the proof of Theorem \ref{proplimqchar}.

Moreover, the coefficients $P_{m',m} \in \mathbb{Z}$ are non-negative. Moreover, using $t$-deformations of the $q$-characters, Nakajima showed (\citep{QVtA}) that the coefficients $P_{m',m}$ can be computed as the evaluation at $t=1$ of some polynomials (analogs of the Kazhdan-Lusztig polynomials for Weyl groups). With this type of formula, one can hope to deduce from the $q$-characters of the standard modules the $q$-characters of the simple modules, which are not known in the general case. 

One would want to have the same type of decomposition in the category $\cO$, in which the simple modules also form a basis. For that, the first step is to build analogs of the standard modules corresponding to each simple module in the category $\cO$.

In order to do that, let us recall that the $q$-characters of the prefundamental representations can be obtained as limits (as formal power series) of sequences of $q$-characters of simple modules. 

In the present paper, we establish that a limit of a particular sequence of $q$-characters of standard modules exist. We conjecture that this limit is the $q$-character of a certain asymptotical standard module. In the case $\mathfrak{g}=\hat{\mathfrak{sl}_{2}}$ we give an explicit formula for the limit and we prove this conjecture. In this case, the simple modules are known for certain natural monoidal subcategories (see \citep[Theorem 7.9]{CABS}), thus we can work the other way and deduce from the simple modules information about the standard modules.
Furthermore, will we show at the end of the paper that the $q$-characters of these asymptotical standard modules satisfy a decomposition formula of the type (\ref{KL}).

This paper is organized as follows. The first section of this paper presents the background for our work. It gathers the definitions and already known properties of the objects we study. We start with the definitions of the quantum affine algebra and the Borel algebra. Section \ref{sectrep} presents some results on the representations of the Borel algebra, specifically the category $\cO$ of Hernandez-Jimbo. Finally, we present the $q$-character theory for the category $\cO$.   

In Section \ref{sectlim}, we recall the definitions of standard modules for finite-dimensional representations, and study more precisely their $q$-characters, which will be the base for the construction of the asymptotical standard modules. We prove the convergence, as a formal power series, of some sequences of normalized $q$-character of standard modules (Theorem \ref{proplimqchar}). Then we conjecture (Conjecture \ref{conjqchar}) that these limits are $q$-characters of certain asymptotical standard modules.

From Section \ref{sectgenstand} on, we focus on the case where $\mathfrak{g}=\hat{\mathfrak{sl}_{2}}$. Section \ref{sectgenstand} tackles the technical part of the construction of the standard modules, with the two main theorems. First, we build a $\Uqbp$-module $T$, with finite-dimensional $\ell$-weight spaces, such that its normalized $q$-character is the limit obtained from the result of Theorem \ref{proplimqchar} (Theorem \ref{theoPsi1}). Then we induce from $T$ a $\Uqb$-module  with finite-dimensional $\ell$-weight spaces and the same $q$-character as $T$ (Theorem \ref{theoPsi1'}).

Finally, in Section \ref{sectdecomp} we present some results for the decomposition of the $q$-characters of our asymptotical standard modules. The last theorem is a result of the type of (\ref{KL}): the limit obtained in Theorem \ref{proplimqchar}, which from Theorem \ref{theoPsi1'} can be realized as the $q$-character of a $\Uqb$-module, admits a decomposition into a sum of $q$-characters of simple modules in the category $\cO$. Moreover, the coefficients are non-negative integers (Theorem \ref{theodecomp}). 
 \vspace{0.4cm}
 
\textit{The author is supported by the European Research Council under the European Union's Framework Programme 
H2020 with ERC Grant Agreement number 647353 Qaffine.}

\section{Background}
	\subsection{The quantum affine algebra and its Borel subalgebra}	

Let us start by recalling the definitions of the two main algebras we study: the quantum affine algebra $\Uqg$ and its Borel subalgebra for the Drinfeld-Jimbo generators $\Uqb$. 	

			\subsubsection{Quantum affine algebra}
		
Let $\mathfrak{g}$ be an untwisted affine Kac-Moody algebra, with $C=(C_{i,j})_{0\leq i,j\leq n}$ its Cartan matrix. Let $\dot{\mathfrak{g}}$ be the associated finite-dimensional simple Lie algebra, let $I=\{1, \ldots, n\}$ be the vertices of its Dynkin diagram and $\dot{C}=(C_{i,j})_{i,j\in I}$ its Cartan matrix.

Let $(\alpha_{i})_{i\in I}$, $(\alpha_{i}^{\vee})_{i\in I}$ and $(\omega_{i})_{i\in I}$ be the simple roots, the simple coroots and the fundamental weights of $\dot{\mathfrak{g}}$, respectively. We use the usual lattices $Q= \bigoplus_{i\in I}\mathbb{Z}\alpha_{i}$, $Q^{+}= \bigoplus_{i\in I}\mathbb{N}\alpha_{i}$ and $P= \bigoplus_{i\in I}\mathbb{Z}\omega_{i}$. Let $P_{\mathbb{Q}}=P\otimes \mathbb{Q}$, endowed with the partial ordering: $\omega \leq \omega'$ if and only if $\omega' - \omega \in Q^{+}$. Let $D=\text{diag}(d_{0}, d_{1}, \ldots, d_{n})$ be the unique diagonal matrix such that $B=DC$ is symmetric and the $(d_{i})_{0\leq i \leq n}$ are relatively prime positive integers. 

Fix an non-zero complex number $q$, which is not a root of unity, and $h\in\mathbb{C}$ such that $q=e^{h}$. Then for all $r\in \mathbb{Q}, q^{r}:=e^{rh}$. Since $q$ is not a root of unity, for $r,s\in \mathbb{Q}$, we have $q^{r}=q^{s}$ if and only if $r=s$. We set $q_{i}:= q^{d_{i}}$, for $0\leq i \leq n$.

We use the following notations.
\begin{equation*}
\begin{array}{ccc}
[m]_{z}=\frac{z^{m}-z^{-m}}{z-z^{-1}},& [m]_{z}! = \prod_{j = 1}^{m}[j]_{z}, & \genfrac[]{0pt}{0}{r}{s}_{z} = \frac{[r]_{z}!}{[s]_{z}![r-s]_{z}!}
\end{array}
\end{equation*}

\begin{defi}
One defines the \textit{quantum affine algebra} $\Uqg$ as the $\mathbb{C}$-algebra generated by $e_{i},f_{i},k_{i}^{\pm 1}, 0\leq i \leq n$, together with the following relations, for $0\leq i,j\leq n$,
\begin{equation*}
\begin{gathered}
k_{i}k_{j}=k_{j}k_{i}, \quad \left[ e_{i},f_{j} \right]= \delta_{i,j}\frac{k_{i}-k_{i}^{-1}}{q_{i}-q_{i}^{-1}}, \\
k_{i}e_{j}k_{i}^{-1}=q_{i}^{C_{i,j}}e_{j}, \quad k_{i}f_{j}k_{i}^{-1}=q_{i}^{-C_{i,j}}e_{j}, \\
\sum_{r=0}^{1-C_{i,j}}(-1)^{r} e_{i}^{(1-C_{i,j}-r)}e_{j}e_{i}^{(r)}  =0,~ (i\neq j), \\
\sum_{r=0}^{1-C_{i,j}}(-1)^{r}f_{i}^{(1-C_{i,j}-r)}f_{j}f_{i}^{(r)} =0,~ (i\neq j) 
\end{gathered}
\end{equation*}
where $x_{i}^{(r)} = x_{i}^{r}/[r]_{q_{i}}!, (x_{i}=e_{i},f_{i})$.
\end{defi}

The algebra $\Uqg$ has another presentation, with the \textit{Drinfeld generators} (\citep{ANRY}, \citep{BGA})
\begin{equation*}
x_{i,r}^{\pm}(i\in I, r\in \mathbb{Z}), \quad \phi_{i,\pm m}^{\pm} (i\in I, m\geq 0),\quad k_{i}^{\pm 1} (i\in I),
\end{equation*}
and some relations we will not recall here, but which are also $q$-deformations of the Weyl and Serre relations.

\begin{ex}
For $\mathfrak{g}=\hat{\mathfrak{sl}_{2}}$, one has the following correspondence 
\begin{align*}
e_{1} &= x_{1,0}^{+}, \enskip f_{1} = x_{1,0}^{-}, \\
e_{0} &= k_{1}^{-1}x_{1,1}^{-}, \enskip f_{0} = x_{1,-1}^{+}k_{1}.
\end{align*}
\end{ex}

Let us introduce the generating series, for $i\in I$ 
\begin{equation*}
\phi_{i}^{\pm}(z) = \sum_{m\geq 0}\phi_{i,\pm m}^{\pm}z^{\pm m} = k_{i}^{\pm 1}\exp\left( \pm (q_{i}-q_{i}^{-1})\sum_{r>0}h_{i,\pm r}z^{\pm r}\right)\quad \in \mathcal{U}_{q}(\mathfrak{h})[z^{\pm 1}].
\end{equation*}
Thus, the $(\phi_{i,\pm m}^{\pm})_{i\in I, m \geq 0}$ and the $(k_{i}^{\pm 1},h_{i,\pm r})_{i\in I, r>0}$ generate the same subalgebra of $\Uqg$: the $\ell$-Cartan subalgebra $\Uqg^0$.

The $(h_{i,\pm r})_{i\in I, r>0}$ can be useful because their relations with the Drinfeld generators are simpler to write. For example, for all $i,j\in I$, $r\in \mathbb{Z}\setminus \{0\}$ and $m\in \mathbb{Z}$,
\begin{equation}\label{relhx}
[ h_{i,r}, x_{j,m}^\pm ] = \pm \frac{[r\dot{C},j]_q}{r}x_{j,r+m}^\pm.
\end{equation}

The quantum affine algebra has a triangular decomposition, associated to the Drinfeld generators: let $\Uqg^{\pm}$ be the subalgebra of $\Uqg$ generated by the $(x_{i,r}^{\pm})_{i\in I, r\in \mathbb{Z}}$. Then (\citep{BGA}),
\begin{equation}\label{trigdecompUqg}
\Uqg \simeq \Uqg^{-}\otimes \Uqg^{0}\otimes \Uqg^{+}.
\end{equation}

The algebra $\Uqg$ has a natural $Q$-grading, with, for all $i\in I, m \in \mathbb{Z}, r>0$,
\begin{equation}\label{Qgrad}
\deg(x_{i,m}^{\pm}) =\pm \alpha_{i}, \quad \deg(h_{i,r}) = \deg(k_{i}^{\pm 1}) = 0.
\end{equation}

		\subsubsection{The Borel subalgebra}

\begin{defi}
The \textit{Borel algebra} $\Uqb$ is the subalgebra of $\Uqg$ generated by the $e_{i}, k_{i}^{\pm 1}$, for $0\leq i\leq n$.
\end{defi}

Let $\Uqb^{\pm} = \Uqg^{\pm}\cap \Uqb$ and $\Uqb^{0} = \Uqg^{0}\cap \Uqb$, then \citep{BCP}
\begin{equation*}
\Uqb^{+} = \langle x_{i,m}^{+}\rangle_{i\in I, m\geq 0}, \quad \Uqb^{0} = \langle \phi_{i,r}^{+}, k_{i}^{\pm}\rangle_{i\in I, r> 0}.
\end{equation*}
\begin{rem}\label{remUqsl2m}
In general, such a nice description does not exist for $\Uqbm$, except when $\mathfrak{g} = \hat{\mathfrak{sl}}_2$. In that case, $\Uqbm$ is isomorphic to the algebra defined by the generators $(x_{1,m}^{-})_{ m\geq 1}$, together with the relations (\citep[Section 4.21]{LQG}), for all $m,l\geq 1$, $i,j \in I$
\begin{equation}\label{eqUqbmg}
x_{i,m+1}^\pm x_{j,l}^\pm -q^{\pm\dot{C}_{i,j}}x_{j,l}^\pm x_{i,m+1}^\pm = q^{\pm\dot{C}_{i,j}}x_{i,m}^\pm x_{j,l+1}^\pm - x_{j,l+1}^\pm x_{i,m}^\pm.
\end{equation}
\end{rem}

We also use the subalgebra $\Uqbp$:
\begin{equation}\label{defUqbp}
\Uqbp := \langle x_{i,m}^{+}, \phi_{i,r}^{+}, k_{i}^{\pm}\rangle_{i\in I, m \geq 0,r> 0}. 
\end{equation}

The triangular decomposition of (\ref{trigdecompUqg}) carries over (\citep{BGA}):
\begin{equation*}
\Uqb \simeq \Uqb^{-}\otimes \Uqb^{0}\otimes \Uqb^{+}.
\end{equation*}

From now on, we are going to consider representations of the Borel algebra $\Uqb$.
\begin{ex}\label{ex1dim}
The algebra $\Uqg$ has only one 1-dimensional representation (of type 1), but $\Uqb$ has an infinite family of one-dimensional representations, indexed by $P_{\mathbb{Q}}$: for each $\omega\in P_{\mathbb{Q}}$, $[\omega]$ denote the one-dimensional representation on which the $(e_{i})_{i\in I}$ act trivially and $k_{i}$ act by multiplication by $q_{i}^{\omega(\alpha_{i})}$.
\end{ex}

It may seem that by studying representations of $\Uqb$ we consider many more representations, but we will see that for finite-dimensional representations, the simple modules are essentially the same.

		\subsubsection{Hopf algebra structure}

The algebra $\Uqg$ has a Hopf algebra structure, where the coproduct and the antipode are given by, for $i\in \{0,\ldots,n\}$,
\begin{equation}\label{HopfUqg}
\begin{gathered}
\Delta(e_{i}) =e_{i}\otimes 1 + k_{i}\otimes e_{i},\\ \Delta(f_{i}) =f_{i}\otimes k_{i}^{-1} + 1\otimes f_{i}, \\ \Delta(k_{i}) = k_{i}\otimes k_{i},\quad S(k_{i}) = k_{i}^{-1} \\
S(e_{i}) = -k_{i}^{-1}e_{i},\quad S(f_{i}) = -f_{i}k_{i}. 
\end{gathered}
\end{equation}
	
With these coproducts and antipodes, the Borel algebra $\Uqb$ is a Hopf subalgebra of $\Uqg$. 

We have the following result for the coproducts in $\Uqb$, where the $Q$-grading follows from the $Q$-grading on $\Uqg$ defined in (\ref{Qgrad}).
\begin{prop}\label{propappcoprod}\citep[Proposition 7.1]{RMAQ} For $r >0$ and $m\in \mathbb{Z}$, 
\begin{equation*}
\Delta(h_{i,r}) \in h_{i,r}\otimes 1 + 1\otimes h_{i,r} + \tilde{\mathcal{U}}_{q}^{-}(\mathfrak{b})\otimes \tilde{\mathcal{U}}_{q}^{+}(\mathfrak{b}),
\end{equation*}
\begin{equation*}
\Delta(x_{i,m}^{+}) \in x_{i,m}^{+}\otimes 1 + \Uqb\otimes (\Uqb X^{+}),
\end{equation*}
where $\tilde{\mathcal{U}}_{q}^{+}(\mathfrak{b})$ (resp. $\tilde{\mathcal{U}}_{q}^{-}(\mathfrak{b})$) is the subalgebra of $\Uqb$ consisting of elements of positive (resp. negative) $Q$-degree, and $X^{+} = \sum_{i\in I,m\in \mathbb{Z}}\mathbb{C}x_{i,m}^{+} \subset \mathcal{U}_{q}^{+}(\mathfrak{b})$.
\end{prop}

These relations of "approximate coproducts" are going to be crucial in the definition of the asymptotical standard modules in Section \ref{sectgendecomp}.

	\subsection{Representations of the Borel algebra}\label{sectrep}
	
In this Section, we recall some results on the representations of $\Uqb$. First of all, we recall the notion of $\ell$-weights and highest $\ell$-weight modules. These notions are at the center of the study of $\Uqb$-modules, as are weights and highest weight modules in the study of representations of semi-simple Lie algebras. Then, we cite some results on the finite-dimensional representations of the Borel algebra. And finally, we recall Hernandez-Jimbo's category $\cO$ for the representations of $\Uqb$.

			\subsubsection{Highest $\ell$-weight modules}\label{HlWM}

Let $V$ be a $\Uqb$-module and $\omega\in P_{\mathbb{Q}}$ a weight. One defines the weight space of $V$ of weight $\omega$ by 
\begin{equation*}
V_{\omega} := \lbrace v\in V \mid k_{i}v=q_{i}^{\omega(\alpha_{i}^{\vee})}v, 0\leq i \leq n \rbrace.
\end{equation*}
The vector space $V$ is said to be \textit{Cartan diagonalizable} if $V=\bigoplus_{\omega\in P_{\mathbb{Q}}}V_{\omega}$.
\begin{defi}
A series $\pmb\Psi = (\psi_{i,m})_{i\in I,m\geq 0}$ of complex numbers, such that $\psi_{i,0}\in q_{i}^{\mathbb{Q}}$ for all $i\in I$ is called an $\ell$\textit{-weight}. The set of $\ell$-weights is denoted by $P_{\ell}$. One identifies the $\ell$-weight $\pmb\Psi$ to its generating series:
\begin{equation*}
\begin{array}{cc}
\pmb\Psi = (\psi_{i}(z))_{i\in I}, & \psi_{i}(z)=\sum_{m\geq 0}\psi_{i,m}z^{m}.
\end{array}
\end{equation*}
\end{defi}
The sets $P_{\mathbb{Q}}$ and $P_{\ell}$ have group structures (the elements of $P_{\ell}$ are invertible formal series) and one has a surjective group morphism $\varpi : P_{\ell} \to P_{\mathbb{Q}}$ which satisfies $\psi_{i}(0)=q_{i}^{\varpi(\pmb\Psi)(\alpha_{i}^{\vee})}$. Let $V$ be a $\Uqb$-module and $\pmb\Psi\in P_{\ell}$ an $\ell$-weight. One defines the $\ell$-weight space of $V$ of $\ell$-weight $\pmb\Psi$ by
\begin{equation*}
V_{\pmb\Psi} := \lbrace v \in V \mid \exists p\geq 0,\forall i \in I, \forall m\geq 0, (\phi_{i,m}^{+} - \psi_{i,m})^{p}v = 0 \rbrace.
\end{equation*} 
A non-zero vector $v\in V$ which belongs to an $\ell$-weight space is called an \textit{$\ell$-weight vector}.
\begin{rem}
As $\phi_{i,0}^{+}=k_{i}$, one has $V_{\pmb\Psi} \subset V_{\varpi(\pmb\Psi)}$.
\end{rem}

\begin{defi}
Let $V$ be a $\Uqb$-module. It is said to be \textit{of highest} $\ell$\textit{-weight} $\pmb\Psi\in P_{\ell}$ if there is $v\in V$ such that $V=\Uqb v$, 
\begin{equation*}
e_{i}v= 0, \forall i \in I \quad \text{ and } \quad \phi_{i,m}^{+}v=\psi_{i,m}v,~\forall i\in I,m \geq 0.
\end{equation*}
In that case, the $\ell$-weight $\pmb\Psi$ is entirely determined by $V$, it is called the $\ell$-weight of $V$, and $v$ is the highest $\ell$-weight vector of $V$.
\end{defi}
\begin{prop}\citep{ARDRF}
For all $\pmb\Psi\in P_{\ell}$ there is, up to isomorphism, a unique simple highest $\ell$-weight module of $\ell$-weight $\pmb\Psi$, denote it by $L(\pmb\Psi)$.
\end{prop}

\begin{ex}\label{ex1dim'}
For $\omega\in P_{\mathbb{Q}}$, the one-dimensional representation defined in Example \ref{ex1dim} is the simple representation $[\omega] = L(\pmb\Psi_{\omega})$, with $(\pmb\Psi_{\omega})_{i}(z) = q_{i}^{\omega(\alpha_{i}^{\wedge})}$, for all $i\in I$.
\end{ex}

Let us define some particular simple modules.
\begin{defi}
For all $i\in I$ and $a\in \mathbb{C}^{\times}$, define the \textit{fundamental representation} $V_{i,a}$ as the simple module $L(Y_{i,a})$, where $Y_{i,a}(z)_{j} = \left\lbrace\begin{array}{ll}
										q_{i}\frac{1-aq_{i}^{-1}z}{1-aq_{i}z} & \text{ if } j=i \\
										1 & \text{ if } j\neq i
							\end{array}\right.$.
\end{defi}
If we generalize this definition:
\begin{defi}
For $i\in I, a\in \mathbb{C}^{\times}$ and $k\geq 0$, the \textit{Kirillov-Reshetikhin module} (or KR-module) $W_{k,a}^{(i)}$ is the simple $\Uqb$-module,
\begin{equation*}
W_{k,a}^{(i)}= L(Y_{i,a}Y_{i,aq_{i}^{2}}\cdots Y_{i,aq_{i}^{2(k-1)}}).
\end{equation*}
\end{defi}
We are going to see in the next section that these are finite-dimensional representations.

Let us define another family of $\ell$-weights. For $i\in I$ and $a\in \mathbb{C}^{\times}$, let $\pmb\Psi_{i,a}^{\pm 1}$ be the $\ell$-weight satisfying
\begin{equation*}
(\psi_{i,a}^{\pm 1})_{j}(z) =  \left\lbrace \begin{array}{ll}
(1-az)^{\pm 1} & \text{ if } i=j,\\
1 & \text{ if not}.
\end{array}\right.
\end{equation*}
Then,
\begin{defi}
For $i\in I$ and $a\in \mathbb{C}^{\times}$, define
\begin{equation*}
L_{i,a}^{\pm} := L(\pmb\Psi_{i,a}^{\pm 1}).
\end{equation*}
The modules $L_{i,a}^{+}$ (resp. $L_{i,a}^{-}$) are the \textit{positives prefundamental representations} (resp. \textit{negative prefundamental representations}).
\end{defi}
The construction of the $(L_{i,a}^{\pm})_{i\in I, a\in \mathbb{C}^{\times}}$ is detailed in \citep{ARDRF}. It is an asymptotical construction, in particular, they are infinite dimensional representations. 

For all $\ell$-weight $\pmb\Psi$, one can consider the \textit{normalized} $\ell$-weight 
\begin{equation}\label{normlweight}
\tilde{\pmb\Psi} = (\varpi(\pmb\Psi))^{-1}\pmb\Psi,
\end{equation}
which is an $\ell$-weight of weight $0$. For example, for $i\in I, a\in \mathbb{C}^{\times}$, $\tilde{Y}_{i,a} = \pmb\Psi_{i,aq_{i}^{-1}}\left(\pmb\Psi_{i,aq_{i}}\right)^{-1}$.

			\subsubsection{Finite-dimensional representations}

Let $\mathscr{C}$ be the category of finite-dimensional $\Uqb$-modules. The $([\omega])_{\omega\in P_{\mathbb{Q}}}$ and the Kirillov-Reshetikhin modules are examples of finite dimensional simple $\Uqb$-modules. 

As stated before, these are not so different from the $\Uqg$-modules. In particular, one has 
\begin{prop}\citep[References for Proposition 3.5]{ARDRF}
Let $V$ be a simple finite-dimensional $\Uqg$-module. Then $V$ is simple as a $\Uqb$-module.
\end{prop} 
Using this result and the classification of finite-dimensional simple module of quantum affine algebras in \citep{GQG}, as well as \citep[remark 3.11]{BRSQIM}, one has 

\begin{prop}\label{propfd}
Let $\pmb\Psi \in P_{\ell}$. Then the simple $\Uqb$-module $L(\pmb\Psi)$ is finite dimensional if and only if, there exists $\omega \in P_\mathbb{Q}$ such that $\pmb\Psi'=\pmb\Psi / \pmb\Psi_\omega$ satisfies: for all $i\in I$, $\Psi'_{i}(z)$ is of the form
\begin{equation*}
\Psi'_{i}(z) = q_{i}^{\deg(P_i)}\frac{P_i(zq_{i}^{-1})}{P_i(zq_{i})},
\end{equation*} 
where $P_i$ are polynomials.

Moreover, in that case, the action of $\Uqb$ can be uniquely extended to an action of $\Uqg$.
\end{prop}

\begin{rem} Equivalently, $L(\pmb\Psi)$ is finite-dimensional if and only if $\pmb\Psi$ is a monomial in the $(Y_{i,a})_{i\in I, a \in \mathbb{C}^{\times}}$. 
In particular, fundamental representations, and more generally KR-modules, are examples of finite-dimensional $\Uqb$-modules.
\end{rem}

\begin{ex}\label{exfond} For $\mathfrak{g}=\hat{\mathfrak{sl}_{2}}$, for all $a\in \mathbb{C}^{\times}$, 
\begin{equation*}
V_{a} = \mathbb{C}v_{a}^{+}\oplus \mathbb{C}v_{a}^{-},
\end{equation*}
with $v_{a}^{+}$ of $\ell$-weight $Y_{a}$ and $v_{a}^{-}$ of $\ell$-weight $Y_{aq^{2}}^{-1}$.

This example will be used later.
\end{ex}

			\subsubsection{Category $\mathcal{O}$}

The category $\mathcal{O}$ of representations of the Borel algebra was first defined in \citep{ARDRF}, mimicking the usual definition of the category $\cO$ BGG for Kac-Moody algebras. Here, we use the definition in \citep{CABS}, which is slightly different. 

For all $\lambda\in P_{\mathbb{Q}}$, define $D(\lambda) := \{ \omega \in P_{\mathbb{Q}} \mid \omega \geq \lambda \}$.
\begin{defi}\label{defcatO}
 A $\Uqb$-module $V$ is in the category $\mathcal{O}$ if
\begin{enumerate}
	\item $V$ is Cartan diagonalizable,
	\item For all $\omega \in P_{\mathbb{Q}}$, one has $\dim(V_{\omega}) < \infty$,
	\item\label{3catO} There is a finite number of $\lambda_{1},\ldots,\lambda_{s}\in P_{\mathbb{Q}}$ such that all the weights that appear in $V$ are in the cone $\bigcup_{j=1}^{s}D(\lambda_{j})$.
\end{enumerate}
\end{defi} 

The category $\cO$ is a monoidal category \citep{ARDRF}.

\begin{ex}
All finite dimensional $\Uqb$-modules are in the category $\mathcal{O}$.  Moreover the positive and negative prefundamental representations are also in the category $\cO$. 
\end{ex}

Let $\Plr$ be the set of $\ell$-weights $\Psi$ such that, for all $i\in I$, $\Psi_{i}(z)$ is rational. We need the following result.

\begin{theo}\citep{ARDRF}
Let $\Psi\in P_{\ell}$. Simple objects in the category $\cO$ are highest $\ell$-weight modules. The simple module $L(\Psi)$ is in the category $\cO$ if and only if $\Psi\in \Plr$. Moreover, if $V$ is in the category $\cO$ and $V_{\Psi}\neq 0$, then $\Psi\in \Plr$.
\end{theo}

	\subsection{A $q$-character theory}

The $q$-characters of finite-dimensional representations of quantum affine algebras were introduced in \citep{QCRQAA} using transfer-matrices. Here, we consider representations in the category $\cO$, which are not necessarily finite-dimensional. Hence we use the $q$-character morphism on the Grothendieck ring of the category $\cO$ defined in \citep{ARDRF}. More precisely, it is the version of \citep{CABS} we use here (since we also use the definition of the category $\cO$ from \citep{CABS}).

After recalling the definition of the $q$-characters, we use it to define some interesting subcategories of the category $\cO$: the categories $\Op, \Om$ and $\OpZ, \OmZ$, as in \citep{CABS}.

			\subsubsection{$q$-characters for category $\mathcal{O}$}

Let $\mathcal{E}_{\ell}$ be the additive group of all maps $c: \Plr \to \mathbb{Z}$ whose support, $\supp(c) = \{ \Psi \in \Plr \mid c(\Psi) \neq 0 \}$ satisfies: $\varpi( \supp(c)) $ is contained in a finite union of sets of the form $D(\mu)$, and, for all $\omega\in P_{\mathbb{Q}}$, the set $\supp(c)\cap \varpi^{-1}(\omega)$ is finite. Similarly, $\mathcal{E}$ is the additive group of maps $c : P_{\mathbb{Q}} \to \mathbb{Z}$ whose support is contained in a finite union of sets of the form $D(\mu)$. The map $\varpi$ naturally extends to a surjective morphism $\varpi : \mathcal{E}_{\ell} \to \mathcal{E}$.

For all $\pmb\Psi \in \Plr$, let $[\pmb\Psi] = \delta_{\pmb\Psi} \in \mathcal{E}_{\ell}$ (resp. $[\omega] = \delta_{\omega} \in \mathcal{E}$, for all $\omega \in P_{\mathbb{Q}}$). 
\begin{rem}
One notices that this notation is coherent with the ones from  Example \ref{ex1dim'}. Indeed, for all $\omega \in P_{\mathbb{Q}}$, the simple one-dimensional representation $[\omega] = L(\pmb\Psi_{\omega})$ is identified with the map $\delta_{\omega} \in \mathcal{E}$.
\end{rem}

\begin{defi}
Let $V$ be a module in the category $\mathcal{O}$. The $q$\textit{-character of} $V$ is the following element of $\mathcal{E}_{\ell}$:
\begin{equation}\label{defqchar}
\chi_{q}(V) = \sum_{\pmb\Psi \in \Plr} \dim(V_{\pmb\Psi})[\pmb\Psi].
\end{equation}
The \textit{character of} $V$ is the following element of $\mathcal{E}$: 
\begin{equation*}
\chi(V) =\varpi(\chi_{q}(V)) = \sum_{\omega \in P_{\mathbb{Q}}}  \dim(V_{\omega})[\omega].
\end{equation*}
\end{defi}

\begin{rem}\citep[Section 3.2]{CABS} In the category $\cO$, every object does not necessarily have finite length. But, as for the category $\cO$ of a classical Kac-Moody algebra (see \citep{IDLA}), the multiplicity of a simple module is well-defined. Hence we have its Grothendieck ring $K_{0}(\mathcal{O})$. Its elements are formal sums, for each $M\in \cO$,
\begin{equation}\label{decompGR}
[M]= \sum_{\pmb\Psi\in P_{\ell}^{\mathfrak{r}}}\lambda_{\pmb\Psi, M}[L(\pmb\Psi)],
\end{equation}
where $\lambda_{\pmb\Psi, M}$ is the multiplicity of $L(\pmb\Psi)$ in $M$. These coefficients satisfy 
\begin{equation*}
\sum_{\pmb\Psi\in P_{\ell}^{\mathfrak{r}},\omega\in P_{\mathbb{Q}}}|\lambda_{\pmb\Psi,M}|\dim((L(\pmb\Psi))_{\omega})[\omega] \quad \in \mathcal{E}.
\end{equation*} $K_{0}(\cO)$ has indeed a ring structure, because of \ref{3catO}, in Definition \ref{defcatO}.
\end{rem}

The $q$\textit{-character morphism} is the group morphism,
\begin{equation*}
\chi_{q} :K_{0}(\mathcal{O}) \mapsto \mathcal{E}_{\ell} 
\end{equation*}
 which sends a class $[V]$ of a representation $V$ to $\chi_{q}(V)$. It is well defined, as $\chi_{q}$ is compatible with exact sequences.
 
\begin{rem}\label{remqchar}
The definition of a $q$-character makes sense for more general representations than that in the category $\cO$. For every representation $V$ with finite-dimensional $\ell$-weight spaces, (\ref{defqchar}) has a sense. However the resulting $q$-character is not necessarily in the ring $\mathcal{E}_{\ell}$. Moreover, the module is not necessarily the sum of its $\ell$-weight spaces (for example, the Verma modules associated to the $\ell$-weights, as in \citep[Section 3.1]{ARDRF}).
\end{rem}

For $V$ a module in the category $\cO$ having a unique $\ell$-weight $\pmb\Psi$ whose weight is maximal, one can consider its normalized $q$-character $\tilde{\chi}_{q}(V)$:
\begin{equation*}
\tilde{\chi}_{q}(V) : = [\pmb\Psi^{-1}]\cdot \chi_{q}(V).
\end{equation*}

For $i\in I$ and $a\in\mathbb{C}^{\times}$, define $A_{i,a}$ as
\begin{equation*}
 Y_{i,aq_{i}^{-1}}Y_{i,aq_{i}}\left( \prod_{\{j\in I\mid C_{j,i}=-1\}}Y_{j,a} \prod_{\{j\in I\mid C_{j,i}=-2\}}Y_{j,aq^{-1}} Y_{j,aq} \prod_{\{j\in I\mid C_{j,i}=-3\}}Y_{j,aq^{-2}} Y_{j,aq}Y_{j,aq^2}   \right)^{-1}.
\end{equation*}
For all $i\in I, a\in\mathbb{C}^{\times}$, $\varpi(A_{i,a})=[\alpha_{i}]$.
\begin{theo}\label{theonormqchar}\citep{QCRQAA, CqC} For $V$ a simple finite-dimensional $\Uqg$-module, one has
\begin{equation*}
\tilde{\chi}_{q}(V) \in \mathbb{Z}[A_{i,a}^{-1}]_{i\in I,a\in\mathbb{C}^{\times}}.
\end{equation*}
\end{theo}

One has more precise results when $V$ is a fundamental representation, we will use some later on.

			\subsubsection{Categories $\Op$ and $\Om$}
	
Let us now recall the definitions of some subcategories of the category $\cO$, introduced in \citep{CABS}. These categories are interesting to study because the $\ell$-weights of the simple modules have some unique decomposition.

\begin{defi}\label{defiPosNegPoids}
An $\ell$-weight of $P_{\ell}^{\mathfrak{r}}$ is said to be \textit{positive} (resp. \textit{negative}) if it is a monomial in the following $\ell$-weights:
\begin{itemize}
	\item[•] the $Y_{i,a} = q_{i}\pmb\Psi_{i,aq_{i}}^{-1}\pmb\Psi_{i,aq_{i}^{-1}}$, where $i \in I$ and $a\in \mathbb{C}^{\times}$,
	\item[•] the $\pmb\Psi_{i,a}$ (resp. $\pmb\Psi_{i,a}^{-1}$), where $i \in I$ and $a\in \mathbb{C}^{\times}$,
	\item[•] the $[\omega]$, where $\omega\in P_{\mathbb{Q}}$.
\end{itemize}
\end{defi}
Let us denote by $P^{+}_{\ell}$ (resp. $P^{-}_{\ell}$) the ring of positive (resp. negative) $\ell$-weights.

\begin{defi}
The category $\Op$ (resp. $\Om$) is the category of representations in $\cO$ whose simple constituents have a positive (resp. negative) highest $\ell$-weight, in the sense of (\ref{decompGR}), that is: for $M$ in $\cO^{\pm}$, one can write,
\begin{equation*}
\chi_{q}(M) = \sum_{\pmb\Psi\in P_{\ell}^{\pm}}\lambda_{\pmb\Psi,M}[L(\pmb\Psi)].
\end{equation*} 
\end{defi}

\begin{rem}\begin{enumerate}[(i)]
	\item The category $\Op$ (resp. $\Om$) contains $\mathscr{C}$, the category of finite-dimensional representations, as well as the positive (resp. negative) prefundamental representations $L_{i,a}^{+}$ (resp. $L_{i,a}^{-}$), for all $i\in I, a\in \mathbb{C}^{\times}$.
	
	\item The generalized Baxter's relations in \citep{BRSQIM} are satisfied in the Grothendieck rings $K_{0}(\cO^{\pm})$.

	\item Positive $\ell$-weights have a unique factorization into a product of $Y_{i,a}$ and $\pmb\Psi_{i,a}$. In particular, for $\mathfrak{g} = \hat{\mathfrak{sl}_{2}}$, this implies a unique factorization of simple modules into products of prime simple representations in $\Op$ (see \citep[Theorem 7.9]{CABS}).
\end{enumerate}
\end{rem}

\begin{theo}\citep{CABS}
The categories $\Op$ and $\Om$ are monoidal categories. 
\end{theo}

		\subsubsection{The categories $\OpZ$ and $\OmZ$}

First, let us recall the infinite quiver defined in \citep[Section 2.1.2]{ACAA}. Let $\tilde{V} = I\times \mathbb{Z}$ and $\tilde{\Gamma}$ be the quiver with vertex set $\tilde{V}$ whose arrows are given by
\begin{equation*}
((i,r) \to (j,s)) \Longleftrightarrow (C_{i,j}\neq 0 \text{ and } s=r+d_{i}C_{i,j}).
\end{equation*}
 
Select one of the two connected components of $\tilde{\Gamma}$ (see \citep[Lemma 2.2]{ACAA}) and call it $\Gamma$. The vertex set of $\Gamma$ is denoted by $V$. 

\begin{ex}\label{exgammasl2}
For $\mathfrak{g}=\mathfrak{sl}_{2}$, the infinite quiver is 
\begin{center}
\begin{tikzpicture}
\draw (-1, -0.5) node{$\tilde{\Gamma}$ : };
\draw (0,0) node{$\bullet$};
\draw (0,0) node[above]{(1,-2)};
\draw (2,0) node{$\bullet$};
\draw (4,0) node{$\bullet$};
\draw (6,0) node{$\bullet$};
\draw (2,0) node[above]{(1,0)};
\draw (4,0) node[above]{(1,2)};
\draw (6,0) node[above]{(1,4)};
\draw (1,-1) node{$\bullet$};
\draw (3,-1) node{$\bullet$};
\draw (5,-1) node{$\bullet$};
\draw (7,-1) node{$\bullet$};
\draw (1,-1) node[below]{(1,-1)};
\draw (3,-1) node[below]{(1,1)};
\draw (5,-1) node[below]{(1,3)};
\draw (7,-1) node[below]{(1,5)};
\draw [->] (0.3,0) -- (1.7,0);
\draw [->] (2.3,0) -- (3.7,0);
\draw [->] (4.3,0) -- (5.7,0);
\draw [->] (1.3,-1) -- (2.7,-1);
\draw [->] (3.3,-1) -- (4.7,-1);
\draw [->] (5.3,-1) -- (6.7,-1);
\draw (0,-1) node{$\cdots$};
\draw [->] (0.3,-1) -- (0.7,-1);
\draw (7,0) node{$\cdots$};
\draw [->] (6.2,0) -- (6.6,0);
\draw (-1,0) node{$\cdots$};
\draw [->] (-0.7,0) -- (-0.3,0);
\draw (8,-1) node{$\cdots$};
\draw [->] (7.2,-1) -- (7.6,-1);
\end{tikzpicture}
\end{center}
In that case, the choice of a connected component is the choice of a parity.
\end{ex}

\begin{defi}\citep{CABS}
Define the category $\OpZ$ (resp. $\OmZ$) as the subcategory of representations of $\Op$ (resp. $\Om$) whose simple components have a highest $\ell$-weight $\pmb\Psi$ such that the roots and poles of $\Psi_{i}(z)$ are of the form $q_{i}^{r}$, with $(i,r) \in V$.
\end{defi}

\begin{rem}\label{remOZ} \begin{enumerate}[(i)]	\item \citep[4.3]{CABS} One does not lose any information by only studying the subcategories $\mathcal{O}_{\mathbb{Z}}^{\pm}$, instead of $\cO^{\pm}$. Indeed, as in the case of finite dimensional representations, each simple object in $\mathcal{O}^{\pm}$ has a decomposition into a tensor product of simple objects which are essentially in $\mathcal{O}_{\mathbb{Z}}^{\pm}$.

	\item Moreover, these categories are in themselves interesting to study, because they are categorification of some cluster algebras (see \citep[Theorem 4.2]{CABS}).
\end{enumerate}
\end{rem}

\section{Limits of $q$-characters of standard modules}\label{sectlim}

In this section, we recall the definition of standard modules for finite-dimensional representations. Then we show that the $q$-characters of a specific sequence of standard modules converge to some limit. We conjecture that this limit is the $q$-character of some $\Uqb$-module which respects some of the structure of the finite-dimensional standard modules.

	\subsection{Standard modules for finite dimensional representations}\label{sectstand}
		
	For finite-dimensional representations of a quantum affine algebra, one can define the \textit{standard module} associated to a given highest $\ell$-weight. Let $m$ be a monomial $m= Y_{i_{1},a_{1}}Y_{i_{2},a_{2}}\cdots Y_{i_{N},a_{N}}$, then the associated standard module is the following tensor product of fundamental representations
	\begin{equation}\label{standstandard}
	M(m):= V_{i_{1},a_{1}}\otimes V_{i_{2},a_{2}}\otimes \cdots\otimes V_{i_{N},a_{N}},
	\end{equation}
where the tensor product is written so that: 
\begin{equation}\label{sstar}
\left( \text{if } k<l \text{, then } a_{l}/a_{k}\notin q^{\mathbb{N}}\right) .
\end{equation} 

Recall the following result, which is a weaker form of \citep[Theorem 5.1 and Corollary 5.3]{BGAT}:
\begin{lemme}\label{lemtwotens}
For $(a,b) \in \left(\mathbb{C}^{\times}\right)^{2}$ such that $a/b \notin q^{\mathbb{Z}}$, and all $i,j\in I$, then
\begin{equation*}
V_{i,a}\otimes V_{j,b} \simeq V_{j,b}\otimes V_{i,a}, \quad \text{ and this module is irreducible.}
\end{equation*}
\end{lemme}

The definition of the standard module $M(m)$ by the expression (\ref{standstandard}) is unambiguous, it does not depend on the order of the factors, as long as (\ref{sstar}) is satisfied. Indeed, with Lemma \ref{lemtwotens}, as long as the condition (\ref{sstar}) is satisfied, two such tensor products are isomorphic.

Moreover, each simple finite-dimensional representation has a highest $\ell$-weight $m$ that can be uniquely written as a monomial in the $Y_{i,a}$'s.

One would want to define analogs of the standard modules for a larger category of representations, namely the category $\mathcal{O}$.

	\subsection{Limits of $q$-character}\label{psi1infty}

It is known (\citep[Theorem 1.1]{tAqCKR} for simply laced types and \citep[Theorem 3.4]{KRC} in general) that the normalized $q$-characters of KR-modules, seen as polynomials in $A_{i,a}^{-1}$, have limits as formal power series. In \citep[Section 6.1]{ARDRF}, the normalized $q$-character of a negative prefundamental representation is explicitly obtained, as a formal power series in $\mathbb{Z}[[A_{i,a}^{-1}]]$,  as a limit of a sequence of normalized $q$-characters of KR-module (hence finite-dimensional representations).

The idea here is to consider the limit of a specific sequence of normalized $q$-characters of finite-dimensional standard modules. The formal power series we obtain is a conjectural normalized $q$-character for a potential infinite-dimensional standard module. 

For $i\in I$ and $a\in \mathbb{C}^{\times}$, consider the negative $\ell$-weight $\pmb\Psi = \pmb\Psi_{i,a}^{-1}$. 

One has $(\pmb\Psi)_{j}=0$ for $j\neq i$ and $\left(\pmb\Psi\right)_{i}(z)  = \frac{1}{1-az}$. Heuristically, one wants to write 
\begin{align*}
\frac{1}{1-az}  &\approx \frac{1-aq_{i}^{-2}z}{1-az}\times\frac{1-aq_{i}^{-4}z}{1-aq_{i}^{-2}z} \times \frac{1-aq_{i}^{-6}z}{1-aq_{i}^{-4}z} \times \frac{1-aq^{-8}z}{1-aq^{-6}z} \times \cdots \\
& \approx \left(\tilde{Y}_{i,aq_{i}^{-1}}\tilde{Y}_{i,aq_{i}^{-3}}\tilde{Y}_{i,aq_{i}^{-5}}\tilde{Y}_{i,aq_{i}^{-7}} \cdots\right)_{i}(z),
\end{align*}
with the normalized $\ell$-weights defined in (\ref{normlweight}).
Thus consider, for $N\geq 1$,
\begin{equation*}
m_{N} := \prod_{k=0}^{N-1} \tilde{Y}_{i,aq_{i}^{-2k-1}}.
\end{equation*}
As stated before, from \citep[Theorem 6.1]{ARDRF}, one knows that, as formal power series, the normalized $q$-characters satisfy
\begin{equation*}
\tilde{\chi_{q}}(L(m_{N})) \xrightarrow[N\to +\infty]{} \tilde{\chi_{q}}(L_{i,a}^{-}).
\end{equation*}

For $N\geq 1$, one can look at the normalized $q$-character of the standard module associated to the same $\ell$-weight $m_{N}$. Consider the standard module
\begin{equation}\label{fdsm}
S_{N}:= V_{i,aq_{i}^{-1}}\otimes V_{i,aq_{i}^{-3}}\otimes V_{i,aq_{i}^{-5}} \otimes \cdots \otimes V_{i,aq_{i}^{-2N+1}}.
\end{equation}
And its normalized $q$-character
\begin{equation*}
\tilde{\chi}_{N}:= \tilde{\chi}_{q}(S_{N}) = \prod_{k=0}^{N-1}\tilde{\chi}_{q}\left(V_{i,aq_{i}^{-2k-1}}\right) = \prod_{k=0}^{N-1}\tilde{\chi}_{q}\left(L(\tilde{Y}_{i,aq_{i}^{-2k-1}})\right).
\end{equation*}

Then 
\begin{theoreme}\label{proplimqchar}
For all $N\geq 1$, $\tilde{\chi}_{N} \in \mathbb{Z}[A_{j,aq^{l}}^{-1}]_{j\in I, l \in \mathbb{Z}}$. 

As a formal power series in $(A_{j,aq^{l}}^{-1})_{j\in I, l \in \mathbb{Z}}$, $\tilde{\chi}_{N}$ has a limit as $N\to +\infty$,
\begin{equation}\label{limqchargen}
\tilde{\chi}_{N} \xrightarrow[N\to +\infty]{} \chi_{i,a}^{\infty} \in \mathbb{Z}[[A_{j,aq^{l}}^{-1}]]_{j\in I, l \in \mathbb{Z}}.
\end{equation}
\end{theoreme}

\begin{rem}\label{remfps}
Here we consider formal sums of monomials in the $(A_{j,aq^{l}}^{-1})_{j\in I, l \in \mathbb{Z}}$:
\begin{equation*}
\mathbb{Z}[[A_{j,aq^{l}}^{-1}]]_{j\in I, l \in \mathbb{Z}} \ni \sum_{\alpha\in \mathbb{N}^{(I\times \mathbb{Z})}} c_{\alpha} A^{-1}_{\alpha},
\end{equation*}
where each $\alpha$ is a finitely supported sequence of non-negative integers on $I\times \mathbb{Z}$, $A_{\alpha} = \prod_{(j,l)\in I\times \mathbb{Z}}A_{j,aq^l}^{\alpha(j,l)}$. This definition is necessary as the result we have to deal with has homogeneous parts (for the standard degree $\deg(A^{-1}_\alpha) = \sum_{(j,l)\in I\times\mathbb{Z}} \alpha(j,l)$) with infinitely many terms, which was not the case in the limits considered in \citep{ARDRF}. For example, here we want to be able to consider elements such as $\sum_{l\in \mathbb{Z}}A_{j,aq^l}^{-1}$, for all $j\in I$.
We get a well-defined ring of formal power series.

The topology on $\mathbb{Z}[[A_{j,aq^{l}}^{-1}]]_{j\in I, l \in \mathbb{Z}}$ is the one of the pointwise convergence: a sequence of its elements converges only if for each monomial $ A^{-1}_{\alpha}$ the corresponding coefficient converges, or more precisely is eventually constant, as these coefficients are integers.
\end{rem}

\begin{proof}
From Theorem \ref{theonormqchar}, for all $N\geq 1$, $\tilde{\chi}_{N} \in \mathbb{Z}[A_{j,b}^{-1}]_{j\in I, b \in \mathbb{C}^{\times}}$. Moreover, by \citep[Lemma 6.1]{CqC}, each $\tilde{\chi}_{q}(V_{i,aq_{i}^{-2k-1}} )$ is in $\mathbb{Z}[A_{j,aq^{l}}^{-1}]_{j\in I, l \in \mathbb{Z}}$, thus $\tilde{\chi}_{N}$ too.

By \citep[proof of Lemma 6.5]{CqC}, for all $j\in I, b \in \mathbb{C}^{\times}$, 
\begin{align*}
\tilde{\chi}_{q}(V_{j,b}) = 1 + \sum m,\text{ with } m \leq A_{j,bq_{j}}^{-1},
\end{align*}
for Nakajima's partial order on the monomials in $(Y_{j,b}^{\pm})_{j\in I, b \in \mathbb{C}^{\times}}$ \citep{QVtA}:
\begin{equation}\label{Nakajimaorder}
m\leq m' \Leftrightarrow m(m')^{-1} = \prod_{\text{finite}} A_{j,b}^{-1}.
\end{equation}

Let us fix a monomial $m \in \mathbb{Z}[A_{j,aq^{l}}^{-1}]_{j\in I, l \in \mathbb{Z}}$. If the factor $\tilde{\chi}_{q}(V_{i,aq_{i}^{-2k-1}})$ contributes to the multiplicity of $m$ in $\tilde{\chi}_{N}$, then $A_{i,aq_{i}^{-2k}}^{-1}$ is a factor of $m$. Thus, only a finite number of factors $\tilde{\chi}_{q}(V_{i,aq_{i}^{-2k-1}})$ can contribute to this multiplicity. In particular, for $k$ large enough, the factor $\tilde{\chi}_{q}(V_{i,aq_{i}^{-2k-1}})$ does not contribute to the multiplicity of $m$ in $\tilde{\chi}_{N}$, for all $N > k$. Thus the multiplicity of $m$ in $\tilde{\chi}_{N}$ is stationary, as $N\to +\infty$. 

As this is true for all monomials in $\mathbb{Z}[A_{j,aq^{l}}^{-1}]_{j\in I, l \in \mathbb{Z}}$, the limit of $\tilde{\chi}_{N}$ as $N \to +\infty$ is well defined as a formal power series (see Remark \ref{remfps}). 
\end{proof}

\begin{ex}
For $\mathfrak{g}=\hat{\mathfrak{sl}_{2}}$, one can compute this formula explicitly. Consider the $\ell$-weight $\pmb\Psi=\pmb\Psi_{1,1}^{-1}$. In this case, the normalized $q$-character of $S_{N}$ is known, 
\begin{equation*}
\tilde{\chi}_{N} =  \prod_{k=0}^{N-1}\left(1 + A_{1,q^{-2k}}^{-1} \right).
\end{equation*}
One can write
\begin{equation*}
\tilde{\chi}_{N}  = \sum_{m=0}^{N-1}\sum_{0\leq k_{1}<k_{2}<\cdots < k_{m} \leq N-1}\prod_{i=1}^{m}A_{1,q^{-2k_{i}}}^{-1}.
\end{equation*}
This formal power series has a limit
\begin{equation}\label{limitqchar}
\tilde{\chi}_{N} \xrightarrow[N\to+\infty]{}  \sum_{m=0}^{+\infty}\sum_{0\leq k_{1}<k_{2}<\cdots < k_{m}}\prod_{i=1}^{m}A_{1,q^{-2k_{i}}}^{-1} \quad \in \mathbb{Z}[[A_{1,q^{-2k}}^{-1}]]_{k\in\mathbb{N}}.
\end{equation}
\end{ex}

\begin{rem}\begin{enumerate}[(i)]\item Except for the weight 0, all weight spaces are infinite-dimensional. Thus this formula is not the (normalized) $q$-character of a representation in the category $\cO$. Nor is the result a formal power series is the "classical" sense. Nonetheless, it can still be a $q$-character, as stated in Remark \ref{remqchar}.

	\item We will see later that this $q$-character also has a decomposition into a sum of $q$-characters of simple representations, as in (\ref{decompGR}).
	\end{enumerate}
\end{rem}

Let us generalize the statement of Theorem \ref{proplimqchar}.
Let $\pmb\Psi$ be a negative $\ell$-weight in $P_{\ell}^{-}$. It can be written as a finite product of $(Y_{i,a})_{i\in I, a \in \mathbb{C}^{\times}}$, $(\pmb\Psi_{i,a}^{-1})_{i\in I, a \in \mathbb{C}^{\times}}$ and $[\omega ]$, with $\omega \in P_{Q}$. Let us write
\begin{equation}\label{decompPsiOm}
\pmb\Psi = [\omega ] \times \prod_{k=1}^{r} Y_{i_{k},a_{k}} \times \prod_{l=1}^{s} \pmb\Psi_{j_{l},b_{l}}^{-1}.
\end{equation}
Each factor $\pmb\Psi_{i,a}^{-1}$ can be seen as a limit of $\prod\tilde{Y}_{i,aq_{i}^{-2k-1}}$. For $N\geq 1$, consider the finite-dimensional standard module $S_{N}$, which is the tensor product of the $V_{i_{k},a_{k}}$, for $1\leq k\leq r$, and the $\overrightarrow{ \bigotimes_{k=0}^{N-1}}V_{j_{l},b_{l}q_{j_{l}}^{-2k-1}}$, for $1\leq l\leq s$, ordered so as to satisfy the condition (\ref{sstar}) of Section \ref{sectstand} (this is a direct generalization of the $S_{N}$ defined in (\ref{fdsm})). 
Then, 
\begin{cor}
The sequence of normalized $q$-characters of $S_{N}$ converges as $N\to +\infty$, as a formal power series. The limit $\chi^{\infty}_{\pmb\Psi}$ can be written
\begin{equation*}
\chi^{\infty}_{\pmb\Psi} = \prod_{k=1}^{r}\left(\tilde{\chi}_{q}(V_{i_{k},a_{k}}) \right)\cdot\prod_{l=1}^{s} \chi_{j_{l},b_{l}}^{\infty}\quad \in \mathbb{Z}[[A_{i,a}^{-1}]]_{i\in I, a \in \mathbb{C}^{\times}}.
\end{equation*} 
\end{cor}

		\subsection{A conjecture}	

As explained in the previous section, the formal power series 
\begin{equation*}
\chi_{\pmb\Psi}^{\infty} = \lim_{N\to + \infty} \tilde{\chi}_{q}(S_{N})
\end{equation*}
is a good candidate for the normalized $q$-character of the asymptotical standard module associated to the negative $\ell$-weight $\pmb\Psi$.

We would like to build, for all negative $\ell$-weight $\pmb\Psi$, a $\Uqb$-module $M(\pmb\Psi)$ with finite-dimensional $\ell$-weight spaces, whose $q$-character is $[\pmb\Psi]\cdot\chi_{\pmb\Psi}^{\infty}$. 
The following conjecture claims that a module with the right $q$-character and which retains some structure of the finite-dimensional standard modules exists.

\begin{conj}\label{conjqchar}
For all negative $\ell$-weight $\pmb\Psi$, there exists a $\Uqb$-module $M(\pmb\Psi)$ with finite-dimensional $\ell$-weight spaces, such that $\chi_{q}(M(\pmb\Psi)) = [\pmb\Psi]\cdot\chi_{\pmb\Psi}^{\infty}$ and the sum of these $\ell$-weights spaces is a $\Uqbp$-module containing a sub-$\Uqb^{+}$-module isomorphic to $S_{N}$ for any $N\geq 0$.
\end{conj}

\begin{rem}
Let us note here that the resulting module is not necessarily equal to the sum of its $\ell$-weight spaces. We will see later that even for $\mathfrak{g}=\hat{\mathfrak{sl}_{2}}$, for certain $\ell$-weights $\pmb\Psi$, there is no $\Uqb$-module $M(\pmb\Psi)$ satisfying the properties of the Conjecture, which is the sum of its $\ell$-weight spaces.
\end{rem}

Section \ref{sectgenstand} will prove this conjecture when $\mathfrak{g}=\hat{\mathfrak{sl}_{2}}$.

	\section{Asymptotical standard modules}\label{sectgenstand}
	
From now on, let us assume that $\mathfrak{g}=\hat{\mathfrak{sl}_{2}}$. We hope in other work to extend these results to the type A and then to other types.

For simplicity, we omit the notation for the node from now on: $Y_{1,a}=Y_{a}$ and $\Psi_{1,a}=\Psi_{a}$.

The aim of this section is to prove Conjecture \ref{conjqchar} for $\mathfrak{g}=\hat{\mathfrak{sl}_{2}}$. We build, for every negative $\ell$-weight $\pmb\Psi$ such that $L(\pmb\Psi)$ is in $\OmZ$, an asymptotical standard $\Uqb$-module $M(\pmb\Psi)$. This module has finite-dimensional $\ell$-weight spaces and its $q$-character is the formal power series $\chi_{\pmb\Psi}^{\infty}$.

The construction will take place in two parts 
\begin{enumerate}
	\item Recall the definition of $\Uqbp$ in (\ref{defUqbp}). First, we build a $\Uqbp$-module $T$, with finite-dimensional $\ell$-weight spaces and the correct $q$-character. 
	\item Then, we build by induction from $T$ a $\Uqb$-module $T^{c}$. We show that $T^{c}$ still has finite-dimensional $\ell$-weight spaces and the same $q$-character.
\end{enumerate}

	\subsection{Construction of the $\Uqbp$-module.}
		\subsubsection{Fundamental example: the case where $\pmb\Psi=\pmb\Psi_{1}^{-1}$}\label{exPsi1}

We build a $\Uqbp$-module $T$ with finite-dimensional $\ell$-weight spaces whose $q$-character is exactly the limit in (\ref{limitqchar}).

For $N\geq 1$, consider the standard module $S_{N}$ as  defined in (\ref{fdsm}), and the normalized standard module $\tilde{S}_{N}$, obtained by normalizing the $\ell$-weights in each factor of $S_{N}$.
\begin{equation}\label{tildeSn}
\tilde{S_{N}} : = L(\tilde{Y}_{q^{-1}})\otimes L(\tilde{Y}_{q^{-3}})\otimes L(\tilde{Y}_{q^{-5}})\otimes \cdots \otimes L(\tilde{Y}_{q^{-2N+1}}).
\end{equation}

As in example \ref{exfond}, each one of the factors of $\tilde{S}_{N}$ is a two dimensional representation
\begin{equation*}
L(\tilde{Y}_{q^{-2r-1}}) = \mathbb{C}v_{r}^{+}\oplus \mathbb{C}v_{r}^{-},
\end{equation*}
where $v_{r}^{+}$ is of $\ell$-weight $\tilde{Y}_{q^{-2r-1}}$ and $v_{r}^{-}$ is of $\ell$-weight $[-2\omega_{1}](\tilde{Y}_{q^{-2r+1}})^{-1}$.

Let $\mathcal{P}_{f}(\mathbb{N})$ be the set of finite subsets of $\mathbb{N}$. Let
\begin{equation}\label{Tpsi1}
T:= \bigoplus_{J\in \mathcal{P}_{f}(\mathbb{N})}\mathbb{C}v_{J}.
\end{equation}

For $J\in \mathcal{P}_{f}(\mathbb{N})$, the element $v_{J}$ represents the infinite simple tensor:
\begin{equation*}
v_{J}= v_{-1}^{\epsilon_{0}(J)}\otimes v_{-3}^{\epsilon_{1}(J)}\otimes v_{-5}^{\epsilon_{2}(J)}\otimes  v_{-7}^{\epsilon_{3}(J)}\otimes \cdots,
\end{equation*}
where $\epsilon_{r}(J) = \left\lbrace \begin{array}{ll}
									- &   \text{if } r\in J, \\
									+ &  \text{if } r\notin J.
\end{array} \right.$
\begin{ex}
For example, $\left\lbrace \begin{array}{l}
				 v_{\emptyset} = v_{-1}^{+}\otimes v_{-3}^{+}\otimes v_{-5}^{+} \otimes v_{-7}^{+}\otimes \cdots \\
				 v_{\{2\}}=  v_{-1}^{+}\otimes v_{-3}^{+}\otimes v_{-5}^{-}\otimes v_{-7}^{+}\otimes \cdots \end{array} \right.$.
\end{ex}

Now, we endow the vector space $T$ with a $\Uqbp$-module structure. The key ingredients are the approximate coproducts formulas of Proposition \ref{propappcoprod}. Let us recall their expressions in this context. For $r >0$ and $m\in \mathbb{Z}$, 
\begin{equation}\label{apprcoprod1}
\Delta(h_{r}) \in h_{r}\otimes 1 + 1\otimes h_{r} + \tilde{\mathcal{U}}_{q}^{-}(\mathfrak{b})\otimes \tilde{\mathcal{U}}_{q}^{+}(\mathfrak{b}),
\end{equation}
\begin{equation}\label{apprcoprod2}
\Delta(x_{m}^{+}) \in x_{m}^{+}\otimes 1 + \Uqb\otimes (\Uqb X^{+}),
\end{equation}
where $\tilde{\mathcal{U}}_{q}^{+}(\mathfrak{b})$ (resp. $\tilde{\mathcal{U}}_{q}^{-}(\mathfrak{b})$) is the subalgebra of $\Uqb$ consisting of elements of positive (resp. negative) degree, and $X^{+} = \sum_{m\in \mathbb{Z}}\mathbb{C}x_{m}^{+} \subset \mathcal{U}_{q}^{+}(\mathfrak{b})$.

Let $J \in \mathcal{P}_{f}(\mathbb{N})$ and $N> i_{0}:=\max (J)$. Consider the truncated simple tensor $v_{J}^{(N)}$, which is an element of the $\Uqb$-module $\tilde{S}_{N}$ defined in (\ref{tildeSn}),
\begin{equation*}
v_{J}^{(N)} := v_{-1}^{\epsilon_{0}(J)}\otimes v_{-3}^{\epsilon_{1}(J)}\otimes v_{-5}^{\epsilon_{2}(J)}\otimes \cdots \otimes  v_{-2N+1}^{\epsilon_{N}(J)}\quad \in \tilde{S}_{N},
\end{equation*}
where $\epsilon_{r}(J) = \left\lbrace \begin{array}{ll}
									- &   \text{if } r\in J, \\
									+ &  \text{if } r\notin J.
\end{array} \right.$

	\paragraph{Action of $\mathcal{U}_{q}(\mathfrak{b})^{+}$:}

The algebra $\mathcal{U}_{q}(\mathfrak{b})^{+}$ is generated by the $(x_{m}^{+})_{m\geq 0}$. 

 For $N> i_{0}$, write 
 \begin{equation}\label{truncvI}
 v_{J}^{(N)} = v_{J}^{(i_{0})}\otimes v_{-2i_{0}-3}^{+}\otimes \cdots \otimes v_{-2N+1}^{+} = v_{J}^{(i_{0})}\otimes u^{(N)},
 \end{equation}
Then, using (\ref{apprcoprod2}), one has, for $m\geq 0$,
  \begin{equation*}
   x_{m}^{+}\cdot  v_{J}^{(N)} = \left(x_{m}^{+}\cdot  v_{J}^{(i_{0})}\right)\otimes u^{(N)} + 0,
  \end{equation*}
  as $\Uqb X^{+}$ acts by $0$ on $u^{(N)}$, which is a highest $\ell$-weight vector and $v_{J}^{(i_{0})}\in \tilde{S}_{i_{0}}$, which is a $\Uqb$-module.
  That way the action of $(x_{m}^{+})_{m\geq 0}$ on $v_{J}$ is defined as 
   \begin{equation}\label{actionxm}
   x_{m}^{+}\cdot  v_{J}:= \left(x_{m}^{+}\cdot  v_{J}^{(i_{0})}\right)\otimes u,
  \end{equation}
 where $u =  v_{-2i_{0}-3}^{+}\otimes \cdots \otimes v_{-2N+1}^{+}\otimes \cdots$.
As $x_{m}^{+}\cdot  v_{J}^{(i_{0})} \in \tilde{S}_{i_{0}}$, it is a linear combination of $v_{K}^{(i_{0})}$, with $\max(K)\leq i_{0}$. Hence,  $x_{m}^{+}\cdot  v_{J}$ is the same linear combination, but with the $v_{K}$ instead of the $v_{K}^{(i_{0})}$.
 
	\paragraph{Action of $\mathcal{U}_{q}(\mathfrak{b})^{0}$:}

The algebra $\mathcal{U}_{q}(\mathfrak{b})^{0}$ is generated by the $(h_{r},k_{1}^{\pm 1})_{r\geq 1}$.  
 
As we normalized the action of $k_{1}$, for all $N>i_{0}$, $k_{1}\cdot v_{\mathcal{I}}^{(N)} = q^{-2|J|}v_{J}^{(N)}$. 
Hence, naturally 
\begin{equation}\label{actionk}
k_{1}\cdot v_{J} := q^{-2|J|}v_{J}.
\end{equation}
For the action of the $h_{r}$'s, let us write, for $N>i_{0}$, $v_{J}^{(N)}$ as in (\ref{truncvI}). Then, using (\ref{apprcoprod1}), one has, for $r\geq 1$, 
\begin{align*}
h_{r}\cdot v_{J}^{(N)}  = \left(h_{r}\cdot v_{J}^{(i_{0})}\right)\otimes u^{(N)}+ v_{J}^{(i_{0})}\otimes \left( h_{r}\cdot u^{(N)}\right)+0, 
\end{align*}
as $\tilde{\mathcal{U}}_{q}^{+}(\mathfrak{b})$ sends $u^{(N)}$ to a higher weight space, which is $\{ 0\}$.

The vector $u^{(N)}\in \tilde{S}_{N}$ is a highest $\ell$-weight vector of $\ell$-weight $\tilde{Y}_{q^{-2i_{0}-3}}\tilde{Y}_{q^{-2i_{0}-5}}\cdots \tilde{Y}_{q^{-2N+1}}$. Hence, $h_{r}\cdot u^{(N)} = \frac{q^{-2i_{0}-2}-q^{-2N}}{q-q^{-1}}u^{(N)}$. 

Thus it is natural to define, for $r\geq 1$, and $N>i_{0}$
\begin{equation}\label{actionhr}
h_{r}\cdot v_{J} := \left(\left(h_{r} + \frac{q^{-2i_{0}-2}}{q-q^{-1}}\id\right)\cdot v_{J}^{(i_{0})}\right)\otimes u.
\end{equation}
As before, $\left(h_{r} + \frac{q^{-2i_{0}-2}}{q-q^{-1}}\id\right)\cdot v_{J}^{(i_{0})} \in \tilde{S}_{i_{0}}$ is a linear combination of $v_{K}^{(i_{0})}$, and $h_{r}\cdot v_{J} $ is the same linear combination of $v_{K}$.

\begin{ex} As $\Delta(h_{1}) = h_{1}\otimes 1 + 1 \otimes h_{1} - (q^{2}-q^{-2})x_{1}^{-}\otimes x_{0}^{+}$, then $h_{1}\cdot \left(v_{-1}^{+}\otimes v_{-3}^{-}\right)  = -q^{-1}(q^{2}-q^{-2})v_{-1}^{-}\otimes v_{-3}^{+}$. Hence,
\begin{equation*}
h_{1}\cdot v_{\{1\}} = \frac{q^{-4}}{q-q^{-1}}v_{\{1\}} - q^{-1}(q^{2}-q^{-2})v_{\{0\}}.
\end{equation*}

Moreover, $h_{1}\cdot v_{\{0\}} = (-q+ \frac{q^{-2}}{q-q^{-1}})v_{\{0\}}$, and $h_{1}(v_{\{1\}}-v_{\{0\}}) = \frac{q^{-4}}{q-q^{-1}}(v_{\{1\}}-v_{\{0\}})$. We will see that $v_{\{0\}}$ is of $\ell$-weight $(\pmb\Psi_{1})^{-1}A_{1}^{-1}$ and $v_{\{1\}}-v_{\{0\}}$ is of $\ell$-weight $(\pmb\Psi_{1})^{-1}A_{q^{-2}}^{-1}$.
\end{ex}

The combination of the last two paragraphs gives us the following result. As the actions of the generators of $\Uqbp$ on $T$ defined in (\ref{actionxm}), (\ref{actionk}) and (\ref{actionhr}) are based on actions on finite-dimensional $\Uqb$-modules, they naturally satisfy the relations in $\Uqbp$.

\begin{prop}
The vector space $T$ has a $\Uqbp$-module structure.
\end{prop}

Moreover, the action of $\Uqbh$ on the $v_{J}$'s is upper-triangular, which allows us to explicit $\ell$-weight vectors. 

\paragraph{Partial order on $\mathcal{P}_{f}(\mathbb{N})$:} 
We define some order on $\mathcal{P}_{f}(\mathbb{N})$. It has nothing to do with Nakajima's partial order on the $\ell$-weights recalled in (\ref{Nakajimaorder}). 
\begin{rem}
For all $J\in \PfN$, $v_{J}$ is a weight vector of weight $[|J|\omega_{1}]$. Moreover, from (\ref{actionhr}), for all $r\geq 1$,
\begin{equation*}
h_{r}\cdot v_{J} \in \bigoplus_{|K|=|J|, K\subset \text{Conv}(J)}\mathbb{C}v_{K},
\end{equation*}
where $\text{Conv}(J)$ is the convex hull of $J$, $\text{Conv}(J) = \left\lbrace j\in \mathbb{N} \mid \min(J) \leq j \leq \max(J) \right\rbrace$.

In particular, sets in $\mathcal{P}_{f}(\mathbb{N})$ correspond to $\ell$-weights which are generically incomparable for Nakajima's partial order.
\end{rem}
\begin{defi}\label{deforder}
For all $N\in \mathbb{N}$, let $\mathcal{P}_{N}(\mathbb{N}) :=\{ J\in \PfN \mid |J| = N \}$. The set $\mathcal{P}_{N}(\mathbb{N})$ is equipped with the lexicographic order on $N$-tuples, noted $\preceq$ .
\end{defi}

\begin{lemme}\label{lemtrig}
For $J\in \PfN$ and $r\geq 1$,
\begin{equation*}
h_{r}\cdot v_{J} \in h_{r,J}v_{J} + \sum_{K\subset \text{Conv}(J), K\prec J}\mathbb{C}v_{K},
\end{equation*}
where the $h_{r,J}$ are the coefficients arising from the action of $h_{r}$ on each component of $v_{J}$.
\end{lemme}

\begin{proof}
Using (\ref{apprcoprod1}) recursively, one has, for $N\geq 1$, 
\begin{multline*}
\Delta^{N}(h_{r})\in \sum_{k=1}^{N} 1\otimes \cdots\otimes \underbrace{h_{r}}_{k}\otimes 1 \otimes \cdots \otimes 1 + 
	\tilde{\mathcal{U}}_{q}^{-}(\mathfrak{b})\otimes \Uqb \otimes \cdots \otimes \Uqb \\+  1\otimes \tilde{\mathcal{U}}_{q}^{-}(\mathfrak{b})\otimes \Uqb \otimes \cdots \otimes \Uqb + \cdots +
1\otimes \cdots\otimes 1 \otimes \tilde{\mathcal{U}}_{q}^{-}(\mathfrak{b})\otimes\tilde{\mathcal{U}}_{q}^{+}(\mathfrak{b}).
\end{multline*}
Hence, using the previous notations, for $N>i_{0}$,
\begin{align*}
h_{r}\cdot v_{J} & = \left(\left(h_{r} + \frac{q^{-2i_{0}-2}}{q-q^{-1}}\id\right)\cdot v_{J}^{(i_{0})}\right)\otimes u 
 \quad \in h_{r,J}v_{J} + \sum_{K\subset \text{Conv}(J), K\prec J}\mathbb{C}v_{K}.
\end{align*}
\end{proof}

\begin{prop}\label{prop1Psi11}
The vector space $T$ has a basis of $\ell$-weights vectors. More precisely, 
\begin{equation*}
T = \bigoplus_{J\in \mathcal{P}_{f}(\mathbb{N})} \mathbb{C}w_{J}, \quad \text{with } w_{J} \text{ of } \ell \text{-weight } (\pmb\Psi_{1})^{-1}\prod_{j\in J}A_{q^{-2j}}^{-1}.
\end{equation*}
As the $\ell$-weight spaces are finite dimensional, one can define a $q$-character for $T$,
\begin{equation}\label{XqT}
\chi_{q}(T) = \sum_{J\in \mathcal{P}_{f}(\mathbb{N})}[\pmb\Psi_{1}^{-1}]\prod_{j\in J}A_{q^{-2j}}^{-1}.
\end{equation}
\end{prop}

\begin{rem}
The normalized $q$-character of $T$ is exactly the limit $\chi_{1,1}^{\infty}$ of the normalized $q$-characters of the sequence of standard modules $(S_{N})_{N\geq 1}$ obtained in (\ref{limitqchar}).
\end{rem}

\begin{proof}
By Lemma \ref{lemtrig}, for all $r\in \mathbb{N}^{*}$, $J\in \mathcal{P}_{f}(\mathbb{N})$, 
\begin{equation*}
h_{r}\cdot v_{J} \in \left( \sum_{m=0}^{+\infty }h_{r}^{J,m}\right)v_{J} + \sum_{K\subset \text{Conv}(J), K\prec J}\mathbb{C}v_{K},
\end{equation*}
where as before $h_{r}^{J,m}$ is the coefficient coming from the action of $h_{r}$ on the $m$-th component of $v_{J}$. 

Hence the action of the $(h_{r})_{r\in \mathbb{N}^{*}}$ is simultaneously diagonalizable. Let us look at the diagonal terms. For $r\geq 1, J \in \PfN$,
\begin{align*}
h_{r}^{J,m} = \left\lbrace\begin{array}{l}
				q^{-(2m+1)r}\frac{[r]_{q}}{r} \text{ if } m\notin J \\
				-q^{-(2m-1)r}\frac{[r]_{q}}{r} \text{ if } m\in J
\end{array}\right. .
\end{align*}
Hence,
\begin{equation*}
\sum_{r=1}^{+\infty}\left(\sum_{m=0}^{+\infty }h_{r}^{J,m}\right)z^{r}  = \frac{1}{q-q^{-1}} \log \left(\frac{1}{1-z}\prod_{m\in J}\frac{1-q^{-2m+2}z}{1-q^{-2m-2}z}\right).
\end{equation*}
Thus the vector space $T$ contains a basis of $\ell$-weight vectors $(w_{J})_{J\in \mathcal{P}_{f}(\mathbb{N})}$, where for each $J\in \mathcal{P}_{f}(\mathbb{N})$, $w_{J}$ is a $\ell$-weight $(\pmb\Psi_{1})^{-1}\prod_{j\in J}A_{q^{-2j}}^{-1}$. Moreover, 
\begin{equation*}
w_{J}\in v_{J}+ \vect(v_{K}, K\preceq J, K\neq J).
\end{equation*}
\end{proof}

\begin{rem}However, the action cannot be extended to the full Borel algebra $\Uqb$ the same way. For example, for $N\in \mathbb{N}$ if we consider the truncated pure tensor vector,
\begin{equation*}
v_{\emptyset}^{(N)} = v_{-1}^{+}\otimes v_{-3}^{+}\otimes v_{-5}^{+} \otimes  \cdots v_{-2N+1}^{+}\otimes v_{-2N-1}^{+},
\end{equation*}
then $x_{1}^-$ acts on $v_{\emptyset}^{(N)} \in L(\tilde{Y}_{q^{-1}})\otimes L(\tilde{Y}_{q^{-3}})\otimes  \cdots \otimes L(\tilde{Y}_{q^{-2N-1}}),$ as
\begin{equation*}
x_{1}^{-}\cdot v_{\emptyset}^{(N)}  = \sum_{k=0}^{N}q^{-2k-1}v_{\{k\}}.
\end{equation*}
Which does not have a limit in $T$ as $N \rightarrow +\infty$.
\end{rem}
\begin{rem}\label{remunique}
Hence, if we consider the Conjecture \ref{conjqchar}, then necessarily, the sum of the $\ell$-weight spaces of any potential asymptotical standard module $M(\pmb\Psi_1^{-1})$ is the infinite tensor product $T$ (it is the only $\Uqbp$-module containing a sub-$\Uqb^{+}$-module isomorphic to $S_{N}$ for any $N\geq 0$). As the action on $T$ cannot be extended, the module $M(\pmb\Psi_1^{-1})$ must be defined differently.
\end{rem}

			\subsubsection{Generalization of this construction}\label{GenConst}

Let $\pmb\Psi \in P_{\ell}^{\mathfrak{r}}$ be a negative $\ell$-weight such that $L(\pmb\Psi)$ is in $\OmZ$. In this section, we generalize Proposition \ref{prop1Psi11} to this context.

As seen in example \ref{exgammasl2}, the roots and poles of $\Psi_{1}(z)$ all have the same parity. Let us write $\pmb\Psi$ as a finite product
\begin{equation*}
\pmb\Psi = [\omega]\times m\times \left(\prod_{r=R_{1}}^{R_{2}}\pmb\Psi_{q^{2r}}^{-b_{r}} \right),
\end{equation*}
where $\omega \in P_{\mathbb{Q}}$, $m$ is a monomial in the $(Y_{q^{2l+1}})_{l\in \mathbb{Z}}$ and $b_{r}\in \mathbb{N}$. 

In Section \ref{psi1infty}, we have seen that each $\pmb\Psi_{q_{2r}}^{-1}$ can be written as an infinite product of $(\tilde{Y}_{q^{2k-1}})_{k\leq r}$. Hence, the $\ell$-weight $\pmb\Psi$ can be seen as a limit of a series of monomials, 
\begin{equation*}
\pmb\Psi \approx \lim_{N\to + \infty}[\omega']\prod_{r=-N}^{R}\tilde{Y}_{q^{2r-1}}^{a_{r}}  \quad\left( = \lim_{N\to + \infty} m_{N}\right),
\end{equation*}
where the sequence $(a_{r})_{-\infty < r\leq R}$ is ultimately stationary. 
\begin{rem}
If $\pmb\Psi$ is only a product of $[\omega]$ and  $(Y_{q^{2l+1}})_{l\in \mathbb{Z}}$ (if $L(\pmb\Psi)$ is finite-dimensional, with Proposition \ref{propfd}), this sequence is stationary. In that case, the vector space $T$ is finite-dimensional. 
\end{rem}
We also know that the sequence of normalized $q$-characters of the standard modules associated to $m_{N}$ converges as $N\to + \infty$. With notations from Section \ref{psi1infty}, 
\begin{equation}\label{limchiqpsi}
\chi_{\pmb\Psi}^{\infty} = \lim_{N\to + \infty}\left( \tilde{\chi}_{q}(S_{N})\right) = \lim_{N\to + \infty}\left( \prod_{r=-N}^{R}\left(1+A_{q^{2r}}^{-1}\right)^{a_{r}}\right) \quad\in \mathbb{Z}[[A_{q^{2r}}^{-1}]]_{r\leq R}.
\end{equation}
\begin{theoreme}\label{theoPsi1}
There exists a $\Uqbp$-module $T_{\pmb\Psi}$, with finite dimensional $\ell$-weight spaces, whose $q$-character is
\begin{equation*}
\chi_{q}(T_{\pmb\Psi})=[\pmb\Psi]\cdot\chi_{\pmb\Psi}^{\infty}.
\end{equation*}
\end{theoreme}

Let us introduce a few notations. Define the set
\begin{equation*}
\mathcal{J}=\lbrace (r,k) \mid r\leq R, 1\leq k\leq a_{r} \rbrace.
\end{equation*}
We consider the following order on $\mathcal{P}_{f}(\mathcal{J})$ (generalized from the order defined in \ref{deforder}).

\begin{defi}\label{deforder2}
Let $N\in \mathbb{N}$ and $J,K \in \mathcal{P}_{f}(\mathcal{J})$ such that $|K|=|J|=N$. 

We say that $J\preceq K$ if and only if they are ordered that way for the lexicographical order on $N$-tuples, while elements of $\mathcal{J}$ are also ordered lexicographically.
\end{defi}

\begin{proof}
Heuristically, the vector space $T_{\pmb\Psi}$  is constructed to be the infinite tensor product
\begin{equation*}
[\omega']\otimes \left(L(\tilde{Y}_{q^{2R-1}})\right)^{\otimes a_{R}}\otimes \left(L(\tilde{Y}_{q^{2R-3}})\right)^{\otimes a_{R-1}}\otimes \cdots \left(L(\tilde{Y}_{q^{-1}})\right)^{\otimes a_{0}}\otimes\cdots \left(L(\tilde{Y}_{q^{-2r-1}})\right)^{\otimes a_{-r}}\otimes \cdots
\end{equation*} 
Hence it is generated by the pure tensors:
\begin{multline}
v_{2R-1}^{\pm^{(1)}}\otimes v_{2R-1}^{\pm^{(2)}}\otimes \cdots v_{2R-1}^{\pm^{(a_{R})}}\otimes v_{2R-3}^{\pm^{(1)}}\otimes \cdots v_{2R-3}^{\pm^{(a_{R-1})}}\otimes \cdots v_{-1}^{\pm^{(1)}}\otimes v_{-1}^{\pm^{(a_{0})}}\otimes \cdots \\
\otimes v_{-2r-1}^{\pm^{(1)}}\otimes \cdots v_{-2r-1}^{\pm^{(a_{-r})}}\otimes\cdots,
\end{multline}
with a finite number of lowest weight components ($-$).

Formally, as in (\ref{Tpsi1}), let $T_{\pmb\Psi}$ be the vector space 
\begin{equation*}
T_{\pmb\Psi}:= \bigoplus_{J\in \mathcal{P}_{f}(\mathcal{J})}\mathbb{C} v_{J}.
\end{equation*}
As before, as the \textit{infinite pure tensors} have a finite number of lowest weight components, and thanks to the approximate coproduct formulas (\ref{apprcoprod1}) and (\ref{apprcoprod2}), the action of $\Uqbp$ on the finite tensor products $\tilde{S}_{N}=[\omega']\bigotimes_{r=-N}^{R}\left(L(\tilde{Y}_{q^{2r-1}})\right)^{\otimes a_{r}}$ stabilizes as $N \rightarrow + \infty$ and the limit can be taken as the action of $\Uqbp$ on $T_{\pmb\Psi}$.

With the decomposition in the proof of Lemma \ref{lemtrig}, one can see that the action of the $\ell$-Cartan subalgebra $\Uqbh$ on the $v_{J}$'s is upper triangular, for the order on $\mathcal{P}_{f}(\mathcal{J})$ defined in (\ref{deforder2}).

Then, one has, for all $J\in \mathcal{P}_{f}(\mathcal{J})$, $r\geq 1$,
\begin{equation*}
h_{r}.v_{J} \in \lambda_{r,J}v_{J} + \bigoplus_{K\prec J, K\subset\text{Conv}(J)} \mathbb{C}v_{K},
\end{equation*}
where the $(\lambda_{r,J})_{r\geq 1}$ satisfy
\begin{equation*}
\sum_{r\geq 1}\lambda_{r,J}z^{r} = \frac{1}{q-q^{-1}} \log(\Phi_{J}(z)),
\end{equation*}
with
\begin{equation*}
\Phi_{J}(z) = \sum_{m\geq 0} \phi_{m,J}z^{m}, \quad \text{and } (\phi_{m,J})_{m\geq 0} = \pmb\Psi\prod_{(r,k)\in J}A_{q^{2r}}^{-1}.
\end{equation*}
Hence the vector space $T_{\pmb\Psi}$ has a basis of $\ell$-weight vectors. Let us write:
\begin{equation*}
T_{\pmb\Psi} = \bigoplus_{J\in \mathcal{P}_{f}(J)} \mathbb{C}w_{J},
\end{equation*}
where, for all $J\in \mathcal{P}_{f}(J)$, $w_{J}$ is an $\ell$-weight vector of $\ell$-weight $\pmb\Psi\prod_{(r,k)\in J}A_{q^{2r}}^{-1}$ (different $w_{J}$ can contribute to the same $\ell$-weight space). 

Thus, $T_{\pmb\Psi}$ has finite dimensional $\ell$-weight spaces and its $q$-character is 
\begin{equation*}
\chi_{q}(T_{\pmb\Psi}) = \sum_{J\in \mathcal{P}_{f}(J)}\pmb\Psi\prod_{(r,k)\in J}A_{q^{2r}}^{-1} = [\pmb\Psi]\cdot\chi_{\pmb\Psi}^{\infty}.
\end{equation*}
\end{proof}

		\subsection{Construction of induced modules}
		
As stated in Remark \ref{remunique}, to obtain a $\Uqb$-module structure on the infinite tensor product, one needs to extend these modules. This is why our asymptotical standard modules will be obtained by induction.

Let $M$ be a $\Uqbp$-module. Define the $\Uqb$-module $M^c$ induced from $M$:
\begin{equation*}
M^c = \Uqb \otimes_{\Uqbp}M \cong \Uqbm\otimes_{\mathbb{C}} M.
\end{equation*}

This induction preserves the $q$-character of the module. More precisely, we have Proposition \ref{proptheo'}, which is obtained from Lemma \ref{lemlweight} bellow.

First of all, we use the following Lemma on the structure of $\Uqbm$.

\begin{lemme}\label{lemUq-}\citep{BPBW}
The elements $\left(x_{m_{1}}^{-}x_{m_{2}}^{-}x_{m_{3}}^{-}\cdots x_{m_{s}}^{-}\right)$, where $s\geq 0$ and $1\leq m_{1}\leq m_{2}\leq m_{3}\leq\cdots\leq m_{s} $, form a basis of $\Uqbm$.
\end{lemme}
\begin{rem}\label{remOre}
This result can also be obtained by seeing the $\left\langle x_m^-\right\rangle_{1\leq m \leq N} \subset\Uqbm$ as successive Ore extensions. Then, with the methods used in \citep{QGK}, we see that the elements $ (x_1^-)^{i_1}(x_2^-)^{i_2}\cdots (x_N^-)^{i_N}$ form a basis of $\left\langle x_m^-\right\rangle_{1\leq m \leq N}$.
\end{rem}
Let us recall relation (\ref{eqUqbmg}) and relation (\ref{relhx}) in this context:
\begin{equation}\label{eqUqbm}
x_{m+1}^- x_l^- -q^{-2}x_l^- x_{m+1}^- = q^{-2}x_m^- x_{l+1}^- - x_{l+1}^- x_m^-, \text{ for all } m,l\geq 1,
\end{equation}
\begin{equation}\label{eqhrxmm}
[h_r,x_m^-] = -\frac{[2r]_q}{r}x_{m+r}^-, \text{ for all } r,m\geq 1.
\end{equation}
Note that the algebra $\Uqbm$ has a natural $\mathbb{N}$-graduation, which is different from the graduation coming from the $Q$-graduation on $\Uqg$ (\ref{Qgrad}). We note:
\begin{equation}\label{Zgrad}
\deg(x_{m_{1}}^{-}x_{m_{2}}^{-}x_{m_{3}}^{-}\cdots x_{m_{s}}^{-}) = \sum_{i=1}^s m_i.
\end{equation}
Thanks to the relation (\ref{eqUqbm}), this is a well-defined graduation on $\Uqbm$.

 \begin{lemme}\label{lemlweight}
 Let $v\in M^c$ be an eigenvector of $h_r$ for a certain $r\geq 1$. Then $v$ belongs to the subspace $1\otimes M$.
 \end{lemme}

\begin{proof}
Let us write, with the result of Lemma \ref{lemUq-}:
\[ v= \sum_{s\geq 0}\sum_{1\leq m_{1}\leq m_{2}\leq\cdots\leq m_{s}}x_{m_{1}}^{-}x_{m_{2}}^{-}\cdots x_{m_{s}}^{-}\otimes u_{m_{1},m_{2},\ldots,m_{s}},
\]
 where all but finitely many of the $u_{m_{1},m_{2},\ldots,m_{s}}\in M$ are 0. By (\ref{eqhrxmm}), one has, for all $s\geq 1$, and all $1\leq m_{1}\leq m_{2}\leq\cdots\leq m_{s}$:
 \[ h_r \left( x_{m_{1}}^{-}x_{m_{2}}^{-}\cdots x_{m_{s}}^{-} \right) = x_{m_{1}}^{-}x_{m_{2}}^{-}\cdots x_{m_{s}}^{-} h_{r} -\frac{[2r]_q}{r}\sum_{j=1}^s x_{m_{1}}^{-}\cdots x_{m_j+r}^-\cdots x_{m_{s}}^{-}.
 \]
 Thus, from the definition of the induced module,
 \begin{multline*}
  h_r\cdot v  = 1\otimes h_r u_{1,\ldots,1} + \sum_{s\geq 1} \left( \sum_{1\leq m_{1}\leq m_{2}\leq\cdots\leq m_{s}} \left( x_{m_{1}}^{-}x_{m_{2}}^{-}\cdots x_{m_{s}}^{-}\otimes h_{r}u_{m_{1},m_{2},\ldots,m_{s}} \right.\right.\\
  \left.\left. -\frac{[2r]_q}{r}\sum_{j=1}^s x_{m_{1}}^{-}\cdots x_{m_j+r}^-\cdots x_{m_{s}}^{-}\otimes u_{m_{1},m_{2},\ldots,m_{s}}\right)\right).
 \end{multline*}
By hypothesis, there exists $\lambda\in\mathbb{C}$ such that $h_r \cdot v = \lambda v$. Thus, from the $Q$-graduation (\ref{Qgrad}) on $\Uqbm$, one has $1\otimes h_r u_{1,\ldots,1}  = \lambda 1\otimes u_{1,\ldots,1}$, and, for all $s\geq 1$,
\begin{multline*}
\sum_{1\leq m_{1}\leq m_{2}\leq\cdots\leq m_{s}} \left( x_{m_{1}}^{-}x_{m_{2}}^{-}\cdots x_{m_{s}}^{-}\otimes h_{r}u_{m_{1},m_{2},\ldots,m_{s}} \right.\\
  \left. -\frac{[2r]_q}{r}\sum_{j=1}^s x_{m_{1}}^{-}\cdots x_{m_j+r}^-\cdots x_{m_{s}}^{-}\otimes u_{m_{1},m_{2},\ldots,m_{s}}\right)\\
   = \lambda \sum_{1\leq m_{1}\leq m_{2}\leq\cdots\leq m_{s}}x_{m_{1}}^{-}x_{m_{2}}^{-}\cdots x_{m_{s}}^{-}\otimes u_{m_{1},m_{2},\ldots,m_{s}}.
\end{multline*}
Moreover, from the $\mathbb{N}$-graduation (\ref{Zgrad}) on $\Uqbm$, one has, for all $s\geq 1$, for all $N\geq 1$,
\begin{multline}\label{grosseeqxmm}
\sum_{\substack{
1\leq m_{1}\leq m_{2}\leq\cdots\leq m_{s} \\
\sum m_i =N +r
}} x_{m_{1}}^{-}x_{m_{2}}^{-}\cdots x_{m_{s}}^{-}\otimes (h_r - \lambda)u_{m_{1},m_{2},\ldots,m_{s}} \\
-\frac{[2r]_q}{r} \sum_{\substack{
1\leq m_{1}\leq m_{2}\leq\cdots\leq m_{s} \\
\sum m_i =N 
}}\sum_{j=1}^s x_{m_{1}}^{-}\cdots x_{m_j+r}^-\cdots x_{m_{s}}^{-}\otimes u_{m_{1},m_{2},\ldots,m_{s}} = 0.
\end{multline}
Let us note that the $x_{m_{1}}^{-}\cdots x_{m_j+r}^-\cdots x_{m_{s}}^{-}$ are not necessarily in the basis of $\Uqbm$ of Lemma \ref{lemUq-}, but from (\ref{eqUqbm}) they can be written as a linear combination of these basis elements. Moreover, with Remark \ref{remOre}, $x_{m_{1}}^{-}\cdots x_{m_j+r}^-\cdots x_{m_{s}}^{-} \in \left\langle x_m^- \right\rangle_{1\leq m\leq \overline{m}} \subset\Uqbm$, where $\overline{m}=\max(m_j+r,m_s)$, and can be decomposed in the basis of this subalgebra (note that $m_j+r$ is possibly greater than $m_s$).

Now, suppose that $v$ does not belong to the subspace $1\otimes M$. Then, for all $s\geq 1$ and all $N\geq 1$ such that there exists $(m_{1},m_{2},\ldots,m_{s})$ with $\sum m_i=N$ and $u_{m_{1},m_{2},\ldots,m_{s}} \neq 0$, consider
\begin{equation}\label{defms}
\overline{m_s}= \max\left\lbrace m_s\in \mathbb{N} \mid \exists s\geq 1,\exists (m_{1},m_{2},\ldots,m_{s-1}), \sum m_i =N,u_{m_{1},\ldots,m_{s-1},\overline{m_{s}}}\neq 0 \right\rbrace .
\end{equation} 

Then, we decompose the left-hand term of relation (\ref{grosseeqxmm}) into sums of pure tensors whose first factor is in the basis of Lemma \ref{lemUq-}. Then we extract from this sum the terms whose first factor is of the form $x_{m_1}^- \cdots x_{m_{s-1}}^- x_{\overline{m_s}+r}^-$. We get a sum equal to 0, from relation (\ref{grosseeqxmm}). Let us explicit the terms we obtain. As $r>0$, the first term of the left-hand-side of (\ref{grosseeqxmm}) does not contribute. The remaining terms have $-\frac{[2r]_q}{r}$ as a scalar factor. The terms obtained from $j=s$ are of the form 
\begin{equation*}
\sum_{\substack{
1\leq m_{1}\leq m_{2}\leq\cdots\leq m_{s-1}\leq \overline{m_s} \\
\sum m_i =N}} x_{m_1}^-\cdots x_{m_{s-1}}^- x_{\overline{m_s}+r}^-\otimes u_{m_1,\ldots,m_{s-1},\overline{m_s}},
\end{equation*}
where the first factors of each terms are elements of the considered basis, so no permutation of the $x_m^-$ is needed. 

By rewriting (\ref{eqUqbm}),
\begin{equation*}
x_{\overline{m_s}+r}^- x_{\overline{m_s}}^-=q^{-2}x_{\overline{m_s}}^- x_{\overline{m_s}+r}^- + \underbrace{q^{-2}x_{\overline{m_s}+r-1}^- x_{\overline{m_s}+1}^- - x_{\overline{m_s}+1}^- x_{\overline{m_s}+r-1}^-}_{\text{does not contribute}},
\end{equation*}
thus, the terms obtained from $j=s-1$ are of the form
\begin{equation*}
\sum_{\substack{
1\leq m_{1}\leq\cdots\leq m_{s-2}\leq m_{s-1}=\overline{m_s} \\
\sum m_i =N}} q^{-2}x_{m_1}^-\cdots x_{m_{s-2}}^- x_{\overline{m_s}}^- x_{\overline{m_s}+r}^-\otimes u_{m_1,\ldots,m_{s-2},\overline{m_s},\overline{m_s}},
\end{equation*}
as only the terms for which $m_{s-1}=\overline{m_s}$ contribute to the considered sum.

Recursively, the terms obtained from j=1 are of the form
\begin{equation*}
q^{-2(s-1)}\left(x_{\overline{m_s}}^-\right)^{s-1}x_{\overline{m_s}+r}^- \otimes u_{\overline{m_s},\overline{m_s},\ldots,\overline{m_s}},
\end{equation*}
if $N=s\overline{m_s}$, and 0 otherwise.
Thus, the resulting formula is:
\begin{equation}
0 = \sum_{\substack{
1\leq m_{1}\leq\cdots\leq m_{s-1}\leq \overline{m_s} \\
\sum m_i =N}} C_{m_1,\ldots,m_{s-1}} x_{m_1}^-\cdots x_{m_{s-1}}^- x_{\overline{m_s}+r}^-\otimes u_{m_1,\ldots,m_{s-1},\overline{m_s}},
\end{equation}
where
\begin{equation}
C_{m_1,\ldots,m_{s-1}} = \sum_{k=0}^{\mathcal{N}} q^{-2k} \text{ , with } \mathcal{N} = \sharp\{ m_i\mid m_i=\overline{m_s} \}.
\end{equation}
As $q$ is not a root of unity, these $C_{m_1,\ldots,m_{s-1}}$ are non-zero. This implies that all the $u_{m_1,\ldots,m_{s-1},\overline{m_s}}$ are 0, which is a contradiction with the hypothesis taken in the definition of $\overline{m_s}$ in (\ref{defms}).
\end{proof}

\begin{rem}
Recall the definition of $\ell$-weight vectors from Section \ref{HlWM}: $v$ in the $\Uqb$-module $V$ is an $\ell$-weight vector if there is an $\ell$-weight $\pmb\Psi$ and $p\in \mathbb{N}$ such that:
\begin{equation*}
\forall i \in I, \forall m\geq 0, (\phi_{i,m}^{+} - \psi_{i,m})^{p}v = 0 .
\end{equation*}
\end{rem}

\begin{prop}\label{proptheo'}
The $\ell$-weights vectors of $M^c$ are exactly the $1\otimes u$, where $u$ is an $\ell$-weight vector of $M$.
\end{prop}

\begin{proof}
First of all, if $u$ is an $\ell$-weight vector of $M$, then by construction of $M^c$, $1\otimes u$ is an $\ell$-weight vector of $M^c$.

Conversely, let us write:
\begin{equation}\label{decompMc}
M^c = \left(1\otimes M\right) \oplus M_{\geq 1},
\end{equation}
where $M_{\geq 1}$ is generated by the pure tensors for which the first factor is of non-zero $Q$-degree. Notice that from (\ref{eqhrxmm}), $M_{\geq 1}$ is a sub-$\Uqbh$-module.

For all $\ell$-weight vector $v$ of $M^c$, there exists:
\begin{equation*}
p_{v} = \min \left\lbrace p\in \mathbb{N}^+ \mid \exists \psi_1 \in \mathbb{C}, (\phi_1^+ - \psi_1)^p v =0 \right\rbrace.
\end{equation*}

Now suppose there exists non-zero $\ell$-weight vectors in $M^c$ which is not in $1\otimes M$. Then its projection on $M_{\geq 1}$ in the decomposition (\ref{decompMc}) is a non-zero $\ell$-weight vector in $M_{\geq 1}$. Consider such a $v_0$ which minimizes $p_v$. Let $p_0=p_{v_0}$. There is $\psi_1\in \mathbb{C}$ such that $(\phi_1^+ - \psi_1)^{p_0} v_0 =0$. By definition,
\begin{equation*}
w = (\phi_1^+ - \psi_1)^{p_0-1} v_0 \neq 0.
\end{equation*}
But, $\phi_1^+ w = \psi_1 w$, as $h_1=k_1^{-1}\phi_1^+$, and $v_0$ and $w$ are weight vectors, then $w$ is an eigenvector of $h_1$.   By Lemma \ref{lemlweight}, $w \in 1\otimes M$. However, $w \in M_{\geq 1}$ as $M_{\geq 1}$ is stabilized by $\Uqbh$. Hence $w=0$, which is a contradiction with the definition of $v_0$.
\end{proof}

Now everything is in place to define the asymptotical standard modules. Let $\pmb\Psi$ be a negative $\ell$-weight such that $L(\pmb\Psi)$ is in $\OmZ$. Define the induced $\Uqb$-module from the $\Uqbp$-module $T_{\pmb\Psi}$ constructed in the previous section:
\begin{equation*}
T_{\pmb\Psi}^{c} := \Uqb \otimes_{\Uqbp} T_{\pmb\Psi}.
\end{equation*}

\begin{theoreme}\label{theoPsi1'}
The vector space $T_{\pmb\Psi}^{c}$ is a $\Uqb$-module, such that its $\ell$-weight spaces are finite-dimensional. As such, $T_{\pmb\Psi}^{c}$ satisfies the Conjecture \ref{conjqchar} for $\pmb\Psi$.  One has:
\begin{equation*}
\chi_{q}(T_{\pmb\Psi}^c)=[\pmb\Psi]\cdot\chi_{\pmb\Psi}^{\infty}.
\end{equation*}
\end{theoreme}

\begin{proof}
From Proposition \ref{proptheo'} we know that the $\ell$-weight vectors of $T_{\pmb\Psi}^c$ are exactly the $1\otimes u$, where $u$ is an $\ell$-weight vector of $T_{\pmb\Psi}$. Thus, with the result of Theorem \ref{theoPsi1}, we know that $T_{\pmb\Psi}^c$ has finite-dimensional $\ell$-weight spaces which satisfy:
\begin{equation*}
\chi_{q}(T_{\pmb\Psi}^c)= \chi_{q}(T_{\pmb\Psi})=[\pmb\Psi]\cdot\chi_{\pmb\Psi}^{\infty}.
\end{equation*}
Furthermore, let us look at the way $T_{\pmb\Psi}^c$ is built to see that it satisfies Conjecture \ref{conjqchar}. Recall the standard modules defined in Section \ref{psi1infty} and used in Section \ref{GenConst}. For all $N\geq 1$, let
	\begin{equation*}
	S_{N}= \overrightarrow{\bigotimes_{r=-N}^{R}}\left( V_{q^{-2r-1}}^{\otimes a_{r}}\right).
	\end{equation*}
Then the sum of $\ell$-weights spaces of $T_{\pmb\Psi}^{c}$ is the $\Uqbp$-module $T_{\pmb\Psi}$, which contains for all $N\geq 1$ a sub-$\Uqb^{+}$-modules isomorphic to $S_{N}$.
\end{proof}

\begin{rem}
The space $T_{\pmb\Psi}^{c}$ contains submodules without $\ell$-weight vectors. For example, for $\pmb\Psi=\pmb\Psi_{1}^{-1}$, the submodule $M=<x_{r}^{-}\otimes w_{\emptyset} - 1\otimes w_{\{0\}}\mid r\in \mathbb{N}^{*}>$ does not contain any $\ell$-weight vectors.
\end{rem}

From now on, for all $\ell$-weights $\pmb\Psi$ such that $L(\pmb\Psi)$ is in the category $\OmZ$, 
\begin{equation*}
M(\pmb\Psi) : = T_{\pmb\Psi}^{c},
\end{equation*}
will denote the generalized standard module associated to the $\ell$-weight $\pmb\Psi$.

	\section{Decomposition of the $q$-character of asymptotical standard modules}\label{sectdecomp}
	
As stated in the introduction, for the category $\mathscr{C}$ of finite-dimensional representations of $\Uqb$, it is known (\citep{QVFDR}) that the classes of the standard modules $[M(m)]$ form a second basis of the Grothendieck ring $K(\mathscr{C})$ (in addition to the classes of simple modules). Moreover, the two bases are triangular with respect to Nakajima's partial ordering of dominant monomials (see \citep{QVtA})):
\begin{equation}\label{KLdf}
[M(m)] = [L(m)] + \sum_{m'< m} P_{m',m}[L(m')],
\end{equation} 
where the coefficients $P_{m',m} \in \mathbb{Z}$ are non-negative.

The aim of this Section is to show that the $q$-characters of the modules we have just built have similar decomposition into sum of $q$-characters of simple modules.

		\subsection{Partial order on $\Plr$}\label{partordr}
		
For a formula of the type (\ref{KLdf}) to make sense, one needs to define a partial order on $\Plr$, which is the index set of both the simple modules and the standard modules in our context. 

We draw our inspiration from the partial order defined in the proof of \citep[Lemma 6.4]{CABS}, which is itself a generalization of the order Nakajima used in \citep{QVtA}, recalled in (\ref{Nakajimaorder}).

\begin{defi}
Let $\pmb\Psi,\pmb\Psi' \in \Plr$. we say that $\pmb\Psi'\leq \pmb\Psi$ if $\pmb\Psi'(\pmb\Psi)^{-1}$ is a monomial in the $A_{i,a}^{-1}$, with $i\in I, a\in \mathbb{C}^{\times}$.
\end{defi}

\begin{rem}
Contrary to the finite-dimensional case, every $\ell$-weight has an infinite number of \textit{lower} $\ell$-weights.
\end{rem}

		\subsection{Decomposition for $M(\pmb\Psi_{1}^{-1})$}

Let us recall the $q$-character of the $\Uqb$-module $M(\pmb\Psi_{1}^{-1})$:
\begin{equation*}
\chi_{q}(M(\pmb\Psi_{1}^{-1})) = \sum_{J\in \mathcal{P}_{f}(\mathbb{N})}[\pmb\Psi_{1}^{-1}]\prod_{j\in J}A_{q^{-2j}}^{-1}.
\end{equation*}

\begin{theoreme}\label{theodecomp}
The $q$-character of $M(\pmb\Psi_{1}^{-1})$ has a decomposition into a sum of $q$-characters of simple modules.  More precisely, 
\begin{equation*}
\chi_{q}(M(\pmb\Psi_{1}^{-1})) = \sum_{m=0}^{+\infty} \sum_{\stackrel{1\leq r_{1}<r_{2}<\cdots < r_{m}}{r_{i+1}>r_{i}+1}} \chi_{q}(L(\pmb\Psi_{1}^{-1}A_{q^{-2r_{1}}}^{-1}A_{q^{-2r_{2}}}^{-1}\cdots A_{q^{-2r_{m}}}^{-1})).
\end{equation*}
\end{theoreme}

\begin{rem}
This formula is multiplicity-free. 
\end{rem}

\begin{proof}
Let $m\in \mathbb{N}^{*}$ and $(r_{1},r_{2},\ldots,r_{m})\in \left(N^{*}\right)^{m}$, satisfying $r_{i+1}>r_{i}+1$ for all $1\leq i \leq m$. One has
\begin{align*}
\pmb\Psi_{1}^{-1}A_{q^{-2r_{1}}}^{-1}A_{q^{-2r_{2}}}^{-1}\cdots A_{q^{-2r_{m}}}^{-1} & =  [(-r_{m}+1)\omega_{1}]\left(Y_{q^{-1}}\cdots Y_{q^{-2r_{1}+3}}\right)\left(Y_{q^{-2r_{1}-3}}\cdots Y_{q^{-2r_{2}+3}}\right) \cdots\\
& \cdots \left(Y_{q^{-2r_{m-1}-3}}\cdots Y_{q^{-2r_{m}+3}}\right) \pmb\Psi_{q^{-2r_{m}-2}}^{-1}, 
\end{align*}
The $q$-sets $\left\lbrace q^{-1}, q^{-3},\ldots, q^{-2r_{1}+3}\right\rbrace$, $\left\lbrace q^{-2r_{i}-3}, q^{-2r_{i}-5},\ldots, q^{-2r_{i+1}+3}\right\rbrace$, for $1\leq i\leq m-1$, and $\left\lbrace q^{-2r_{m}-1}, q^{-2r_{m}-3} , \ldots \right\rbrace$ are pairwise in general position. Hence, following \citep[Theorem 7.9]{CABS}, the following tensor product is simple:
\begin{equation*}
L(Y_{q^{-1}}\cdots Y_{q^{-2r_{1}+3}})\otimes L(Y_{q^{-2r_{1}-3}}\cdots Y_{q^{-2r_{2}+3}}) \otimes \cdots \otimes L(Y_{q^{-2r_{m-1}-3}}\cdots Y_{q^{-2r_{m}+3}})\otimes L(\pmb\Psi_{q^{-2r_{m}-2}}^{-1})
\end{equation*}
and of highest $\ell$-weight $\pmb\Psi_{1}^{-1}A_{q^{-2r_{1}}}^{-1}A_{q^{-2r_{2}}}^{-1}\cdots A_{q^{-2r_{m}}}^{-1}$.
Thus, 
\begin{multline*}
L(\pmb\Psi_{1}^{-1}A_{q^{-2r_{1}}}^{-1}A_{q^{-2r_{2}}}^{-1}\cdots A_{q^{-2r_{m}}}^{-1}) = [(-r_{m}+1)\omega_{1}] \otimes L(Y_{q^{-1}}\cdots Y_{q^{-2r_{1}+3}})\otimes \\
L(Y_{q^{-2r_{1}-3}}\cdots Y_{q^{-2r_{2}+3}}) \otimes \cdots \otimes L(Y_{q^{-2r_{m-1}-3}}\cdots Y_{q^{-2r_{m}+3}})\otimes L(\pmb\Psi_{q^{-2r_{m}-2}}^{-1}).
\end{multline*}
One can then compute the $q$-character of $L(\pmb\Psi_{1}^{-1}A_{q^{-2r_{1}}}^{-1}A_{q^{-2r_{2}}}^{-1}\cdots A_{q^{-2r_{m}}}^{-1})$.
\begin{align*}
[L(\pmb\Psi_{1}^{-1}A_{q^{-2r_{1}}}^{-1}A_{q^{-2r_{2}}}^{-1}\cdots A_{q^{-2r_{m}}}^{-1})] & = [(-r_{m}+1)\omega_{1}] [L(Y_{q^{-1}}\cdots Y_{q^{-2r_{1}+3}})] \\
  [L(Y_{q^{-2r_{1}-3}}\cdots Y_{q^{-2r_{2}+3}})]& \cdots [L(Y_{q^{-2r_{m-1}-3}}\cdots Y_{q^{-2r_{m}+3}})] [L(\pmb\Psi_{q^{-2r_{m}-2}}^{-1})] \\
&  = [\pmb\Psi_{1}^{-1}] \sum_{\stackrel{I\in \mathcal{P}_{f}(\mathbb{N})}{(\star)}}\prod_{j\in J}A_{q^{-2j}}^{-1}. 
\end{align*}
where in $(\star)$ we consider the finite sets of $\mathbb{N}$ with $m$  connected components, starting with $r_{1},r_{2},\ldots, r_{m}$ respectively, and with $m+1$ connected components, starting with $0,r_{1},r_{2},\ldots, r_{m}$ respectively.

Thus, 
\begin{multline*}
\sum_{m=0}^{+\infty} \sum_{\stackrel{1\leq r_{1}<r_{2}<\cdots < r_{m}}{r_{i+1}>r_{i}+1}} [L(\pmb\Psi_{1}^{-1}A_{q^{-2r_{1}}}^{-1}A_{q^{-2r_{2}}}^{-1}\cdots A_{q^{-2r_{m}}}^{-1})] \\
	= [\pmb\Psi_{1}^{-1}] \sum_{J\in \mathcal{P}_{f}(\mathbb{N})}\prod_{j\in J}A_{q^{-2j}}^{-1} = \chi_{q}(M(\pmb\Psi_{1}^{-1}))
\end{multline*}
\end{proof}

\begin{rem}
One has indeed
\begin{equation*}
\chi_{q}(M(\pmb\Psi_{1}^{-1})) \in \sum_{\pmb\Psi\leq \pmb\Psi_{1}^{-1}} \mathbb{N}[L(\pmb\Psi)],
\end{equation*}
for the partial order on $P_{\mathfrak{r}}^{\ell}$ defined in Section \ref{partordr}.
\end{rem}

		\subsection{General decomposition}\label{sectgendecomp}

Consider a negative $\ell$-weight $\pmb\Psi$, as in Section \ref{GenConst}, such that $L(\pmb\Psi)$ is in $\OmZ$. It can be written as a finite product

\begin{equation}\label{psi}
\pmb\Psi  = [\omega]\times m \times \prod \pmb\Psi_{q^{2r}}^{-v_{r}},
\end{equation}
where $\lambda\in P_{\mathbb{Q}}$, $m$ is a monomial in $\mathbb{Z}[Y_{q^{2r-1}}]_{r\in \mathbb{Z}}$ and $v_{r}\in  \mathbb{N}$.
First we need the following Lemma.
\begin{lemme}
For $\pmb\Psi^{1},\pmb\Psi^{2}$, negative $\ell$-weights as in (\ref{psi}), one has 
\begin{equation}\label{multqchar}
\chi_{q}(M(\pmb\Psi^{1}\pmb\Psi^{2})) = \chi_{q}(M(\pmb\Psi^{1}))\chi_{q}(M(\pmb\Psi^{2})).
\end{equation}
\end{lemme}

\begin{proof}
If $\pmb\Psi^{1}$ and $\pmb\Psi^{2}$ are monomials in $\mathbb{Z}[Y_{q^{2r-1}}]_{r\in \mathbb{Z}}$, then the standard modules $M(\pmb\Psi^{1})$ and $M(\pmb\Psi^{2})$ are finite dimensional and the result is known.

Else, as in (\ref{limchiqpsi}), their normalized $q$-characters are limits of normalized $q$-characters of finite dimensional standard modules:
\begin{align*}
\tilde{\chi}_{q}(M(\pmb\Psi^{i})) = \lim_{N\to+\infty} \tilde{\chi}_{q}(M(m^{i}_{N})), \forall i \in\{1,2\}.
\end{align*}
The result is obtained by taking the products of the limits.
\end{proof}

Thus, the $q$-character of the standard module $M(\pmb\Psi)$ is a product,
\begin{align*}
\chi_{q}(M(\pmb\Psi)) = [\omega]\cdot \chi_{q}(m) \cdot \prod \chi_{q}(\pmb\Psi_{a})^{-v_{a}},
\end{align*}
for which each term has a decomposition into a sum of $q$-characters of simple modules, corresponding to \textit{lower} highest $\ell$-weights. 

Finally, using the following lemma, which is straightforward from the definition of the order.
\begin{lemme}
The order on the negative $\ell$-weights is compatible with the product. More precisely, for $\pmb\Psi^{1},\pmb\Psi^{2},\pmb\Psi$ and $\pmb\Psi'$ some negative $\ell$-weights,
\begin{align*}
\left( \pmb\Psi\leq \pmb\Psi^{1} \text{ and } \pmb\Psi'\leq \pmb\Psi\pmb\Psi^{2}\right) \Rightarrow \pmb\Psi'\pmb\leq \pmb\Psi^{1}\pmb\Psi^{2}.
\end{align*}
\end{lemme}

One can finally conclude,
\begin{cor}
For every negative $\ell$-weight $\pmb\Psi$ such that $L(\pmb\Psi)$ is in $\OmZ$, one has
\begin{equation}\label{decominfinite}
\chi_{q}(M(\pmb\Psi)) \in \sum_{\pmb\Psi'\leq \pmb\Psi} \mathbb{N}\chi_{q}(L(\pmb\Psi')).
\end{equation}
\end{cor}

\begin{rem}
We showed that the coefficients in the decomposition (\ref{decominfinite}) are non-negative. It would be interesting to show that these coefficients can be interpreted as dimensions, which would explain their non-negativity. 
\end{rem}

\bibliographystyle{alpha}
\bibliography{article}

\flushleft{
\textsc{Université Paris-Diderot, \\
CNRS Institut de mathématiques de Jussieu-Paris Rive Gauche, UMR 7586,} \\
Bâtiment Sophie Germain, Boite Courrier 7012, \\
8 Place Aurélie Nemours - 75205 PARIS Cedex 13,\\
E-mail: \texttt{lea.bittmann@imj-prg.fr}
}

\end{document}